\documentclass[11pt, reqno]{amsart}
\usepackage{indentfirst, amssymb, amsmath, amsthm, mathrsfs, setspace, indentfirst, enumerate, mathrsfs, amsmath, amsthm}
\usepackage[bookmarksnumbered, colorlinks, plainpages]{hyperref}
\usepackage{mathrsfs}

\newtheorem*{theo2A}{Theorem 2.A}
\newtheorem*{theo2B}{Theorem 2.B}

\newtheorem*{theo3A}{Theorem 3.A}

\newtheorem{theo}{Theorem}[section]
\newtheorem{lem}{Lemma}[section]
\newtheorem{cor}{Corollary}[section]

\newcommand{\ol}{\overline}

\newcommand{\be}{\begin{equation}}
\newcommand{\ee}{\end{equation}}
\newcommand{\beas}{\begin{eqnarray*}}
\newcommand{\eeas}{\end{eqnarray*}}
\newcommand{\bea}{\begin{eqnarray}}
\newcommand{\eea}{\end{eqnarray}}

\numberwithin{equation}{section}
\begin{document}

\title[E\MakeLowercase{stimation for the logarithmic partial derivative and}......]{\LARGE E\Large\MakeLowercase{stimation for the logarithmic partial derivative and} W\MakeLowercase{iman}-V\MakeLowercase{aliron theory in several complex variables}}

\date{}
\author[J. F. X\MakeLowercase{u}, S. M\MakeLowercase{ajumder} \MakeLowercase{and} N. S\MakeLowercase {arkar}]{J\MakeLowercase{unfeng} X\MakeLowercase{u}, S\MakeLowercase{ujoy} M\MakeLowercase{ajumder}$^*$ \MakeLowercase{and} N\MakeLowercase{abadwip} S\MakeLowercase{arkar}}
\address{Department of Mathematics, Wuyi University, Jiangmen 529020, Guangdong, People's Republic of China.}
\email{xujunf@gmail.com}
\address{Department of Mathematics, Raiganj University, Raiganj, West Bengal-733134, India.}
\email{sm05math@gmail.com, sjm@raiganjuniversity.ac.in}
\address{Department of Mathematics, Raiganj University, Raiganj, West Bengal-733134, India.}
\email{naba.iitbmath@gmail.com}

\renewcommand{\thefootnote}{}
\footnote{*\emph{Corresponding Author}: Sujoy Majumder.}

\renewcommand{\thefootnote}{\arabic{footnote}}
\setcounter{footnote}{0}

\begin{abstract} The first objective of the paper is to estimate logarithmic partial derivative for meromorphic functions in several complex variables. Our estimations for logarithmic partial derivatives extend the results of Gundersen \cite{GG2} to the higher dimensions. The second objective of the paper is to study Wiman-Valiron theory into higher dimensions. The third objective of this paper is to study the growth of solutions of a class of complex partial differential equations and our obtained result is the extension of the result of Cao \cite{c2} from $\mathbb{C}$ to $\mathbb{C}^m$.
\end{abstract}

\thanks{Typeset by \AmS -\LaTeX}
\maketitle

\section{{\bf Preliminaries on Nevanlinna Theory in $\mathbb{C}^m$}}
We first recall some of the standard notation of Nevanlinna theory in several complex variables (see \cite{HLY1,MR1, WS1}).
We define $\mathbb{Z}_+=\mathbb{Z}[0,+\infty)=\{n\in \mathbb{Z}: 0\leq n<+\infty\}$ and $\mathbb{Z}^+=\mathbb{Z}(0,+\infty)=\{n\in \mathbb{Z}: 0<n<+\infty\}$.
As on any complex manifold $M$, the exterior derivative $d$ splits
\[\displaystyle d= \partial+ \bar{\partial}\]
and twists to
\[\displaystyle d^c= \frac{\iota}{4\pi}\left(\bar{\partial}- \partial\right).\]

Clearly $dd^{c}= \frac{\iota}{2\pi}\partial\bar{\partial}$. A non-negative function $\tau:M\to \mathbb{R}[0,b)\;(0<b\leq \infty)$ of class $\mathbb{C}^{\infty}$ is said to be an exhaustion of $M$ if $\tau^{-1}(K)$ is compact whenever $K$ is.

\smallskip
A hermitian vector space $W$ is a complex vector space together with a positive definite hermitian form on $W$. The hermitian product of $a\in W$ and $b\in W$ is denoted by $(a|b)$. The norm of $a$ is $||a||=\sqrt{(a|a)}$. If $W=\mathbb{C}^m$, then the hermitian metric shall always be given by $(a|b)=\sideset{}{_{i=1}^m}{\sum} a_i\ol b_i$, where $a=(a_1,a_2,\ldots,a_m)\in\mathbb{C}^m$ and $b=(b_1,b_2,\ldots,b_m)\in\mathbb{C}^m$.

\smallskip
An exhaustion $\tau$ of $W$ is defined by $\tau(z)=||z||^2$. The standard Kaehler metric on $W$ is given by $\upsilon=dd^c\tau>0$.
On $W\backslash  \{0\}$, we define
$\displaystyle \omega=dd^c\log \tau\geq 0$ and $\sigma=d^c\log \tau \wedge \omega^{m-1}$,
where $m=\dim(W)$. We define
\[\displaystyle B(a,r)=\{z\in W: ||z-a||<r\},\;\;B(r)=\{z\in W: ||z||<r\},\;\; R(a,r,s)=\ol B(a,r)-B(a,s),\]
\[\displaystyle R(r,s)=R(0,r,s),\;\;S(a,r)=R(a, r, r)\;\;\text{and}\;\; S(r)=S(0,r).\]

\smallskip
Let $U\subset W$ be an open set. A closed set $A\subset U$ is said to be analytic, if for each $a\in A$, there are a finite number of holomorphic functions $f_1,f_2,\ldots,f_l$ defined in a neighbourhood $N(a)$ of $a$ such that (see \cite[pp. 42]{NW}) $A\cap N(a)=\{z\in N(a):f_1(z)=f_2(z)=\cdots=f_l(z)=0\}$.

\smallskip
If $f_j(z),1\leq j\leq l$ can be taken so that their differentials at $a\in A$, $df_1(a), df_2(a), \ldots,df_l(a)$ are linearly independent, then $a$ is called a regular point of $A$. A point of $A$ which is not regular called a singular point. The subset of all singular points of $A$ is denoted by $S(A)$. Set $R(A)=A\backslash  S(A)$. If $S(A)=\phi$, $A$ is said to be regular.

\smallskip
Let $R(A)=\bigcup_{\lambda}A_{\lambda}'$ be the decomposition into the connected components. Then the closure $A=\ol{A'_{\lambda}}$ is analytic and $A=\bigcup_{\lambda}A_{\lambda}$ is a locally finite covering. Set $\dim(A)=\max_a \dim_a(A)$, where $\dim_a(A)=\max_{a\in A_{\lambda}}\dim(A_{\lambda})$. If $\dim_a(A)=\dim(A)$ at all points $a\in A$, then $A$ is said to be of pure dimension $\dim(A)$.

\medskip
Abbreviate $\mathbb{P}^m=\mathbb{P}(\mathbb{C}^{m+1})$, the projective space. Identify the compactified plane $\mathbb{C}\cup\{\infty\}$ with $\mathbb{P}^1=\mathbb{P}(\mathbb{C}^2)$ by setting 
\[\mathbb{P}(z,w)=z/w,\;\;\mathbb{P}(z,1)=z\;\;\text{and}\;\;\mathbb{P}(w,0)=\mathbb{P}(1,0)=\infty\]
if $z\in\mathbb{C}$ and $w\in\mathbb{C}\backslash \{0\}$.
Let $U\neq \varnothing$ be an open connected subset of $\mathbb{C}^m$. Let $\mathcal{O}(U,\mathbb{C}^2)$ be the set of all holomorphic vector functions $u:U\to \mathbb{C}^2$. Define $\mathcal{O}_*(U,\mathbb{C}^2)=\mathcal{O}(U,\mathbb{C}^2)\backslash \{0\}$. Two holomorphic vector functions $u, v\in \mathcal{O}_*(U,\mathbb{C}^2)$ are called equivalent if $u=\lambda v$ for some $\lambda\in \mathbb{C}$. This defines an equivalence relation on $\mathcal{O}_*(U, \mathbb{C}^2)$. An equivalence class $f$ is said to be a meromorphic map from $U$ into $\mathbb{P}^1$ and each $u$ is said to be a representation of $f$ (see \cite[pp. 630]{WS3}).
Let $u$ and $\tilde u$ be reduced representations of the meromorphic map $f$. Then the indeterminacy
\[\displaystyle I_f=\{z\in U: u(z)=0\}=\{z\in U: \tilde u(z)=0\}\]
is well defined analytic set, with $\dim I_f\leq  m-2$. If $z\in U-I_f$, then
\[\displaystyle f(z)=\mathbb{P}(u(z))=\mathbb{P}(\tilde u(z))\in \mathbb{P}(\mathbb{C}^2)=\mathbb{P}^1\]
is well defined and the map $f:U-I_f\to \mathbb{P}(\mathbb{C}^2)=\mathbb{P}^1$ is holomorphic. Therefore a holomorphic vector function $F: U\rightarrow \mathbb{C}^2$ is said to be a representation of $f$ on $U$ if $F\not\equiv 0$ and if
\[\displaystyle f(z)=\mathbb{P}(F(z)),\;\;z\in U-F^{-1}(\{0\}).\]

The representation is said to be reduced if $\dim F^{-1}(\{0\})\leq  m-2$. Now for a non-constant meromorphic function $f:\mathbb{C}^m\to \mathbb{P}^1$ such that $f^{-1}(\{\infty\})\neq \mathbb{C}^m$, let $F=(g,h):\mathbb{C}^m\rightarrow \mathbb{C}^2$ be a reduced representation of $f$. Then
\[\displaystyle f(z)=\mathbb{P}(F(z))=g(z)/h(z),\;\;z\in \mathbb{C}^m-F^{-1}(\{0\}).\]

Clearly $I_f=\{z\in \mathbb{C}^m: g(z)=h(z)=0\}$ and $\dim\{z\in \mathbb{C}^m: g(z)=h(z)=0\}\leq m-2$.

\smallskip
Let $G\not=\varnothing$ be an open subset of $W$, hermitian vector space of $\dim(W)=m\geq 1$. Let $f$ be a holomorphic function on $G$ and $I=(i_1,\ldots,i_m)\in\mathbb{Z}^m_+$ be a multi-index. We define
\[\displaystyle \partial^{I}(f(z))=\frac{\partial^{|I|}f(z)}{\partial z_1^{i_1}\cdots \partial z_m^{i_m}},\]
where $|I|=\sum_{j=1}^m i_j$. Then we can write $f(z)=\sum_{i=0}^{\infty}P_i(z-a)$, where the term $P_i(z-a)$ is either identically zero or a homogeneous polynomial of degree $i$. Certainly the zero multiplicity $\mu^0_f(a)$ of $f$ at a point $a\in G$ is defined by $\mu^0_f(a)=\min\{i:P_i(z-a)\not\equiv 0\}$.

\smallskip
Let $f$ be a meromorphic function on $G$. Then there exist holomorphic functions $g$ and $h$ such that $hf=g$ in $\mathbb{C}^m$ and $g$ and $h$ are coprime at every point of $\mathbb{C}^m$ \cite[pp. 186]{LH1}. Therefore the $c$-multiplicity of $f$ is just $\mu^c_f=\mu^0_{g-ch}$ if $c\in\mathbb{C}$ and $\mu^c_f=\mu^0_h$ if $c=\infty$. The function $\mu^c_f: \mathbb{C}^m\to \mathbb{Z}$ is nonnegative and is called the $c$-divisor of $f$. If $f\not\equiv 0$ on each component of $G$, then $\nu=\mu_f=\mu^0_f-\mu^{\infty}_f$ is called the divisor of $f$. We define $\displaystyle \text{supp}\; \nu=\ol{\{z\in G: \nu(z)\neq 0\}}$.
Clearly $\text{supp}\; \nu$ is either empty or an analytic set of pure dimension $m-1$.

For $t>0$, the counting function $n_{\nu}$ is defined by
\beas \displaystyle n_{\nu}(t)=t^{-2(m-1)}\int_{A\cap B(t)}\nu\; \upsilon^{m-1},\eeas
where $A=\text{supp}\;\nu$. $n_{\nu}$ increases if $\nu$ is non-negative.
The valence function of $\nu$ is defined by
\[\displaystyle N_{\nu}(r)=N_{\nu}(r,r_0)=\int_{r_0}^r n_{\nu}(t)\frac{dt}{t}\;\;(r\geq r_0).\]

For $a\in\mathbb{P}^1$, we write $n_{\mu_f^a}(t)=n(t,a;f)$ if $a\in\mathbb{C}$ and $n_{\mu_f^a}(t)=n(t,f)$ if $a=\infty$. Also $N_{\mu_f^a}(r)=N(r,a;f)$ if $a\in\mathbb{C}$ and $N_{\mu_f^a}(r)=N(r,f)$ if $a=\infty$.

\medskip
Take $0<R\leq +\infty$. Let $f\not\equiv 0$ be a meromorphic function on $W(R)$. Let $0<r_0<r<R$. Now with the help of the positive logarithm function, we define the proximity function of $f$ by
\[\displaystyle m(r, f)=\int_{S(r)} \log^+ |f|\;\sigma \geq  0.\]

The characteristic function of $f$ is defined by $T(r,f)=m(r,f)+N(r,f)$. We know that
\bea\label{la}\displaystyle T\left(r,\frac{1}{f}\right)=T(r,f)-\int_{S(r_0)} \log |f|\;\sigma.\eea

We define $m(r,a;f)=m(r,f)$ if $a=\infty$ and $m(r,a;f)=m\big(r,\frac{1}{f-a}\big)$ if $a$ is finite complex number. Now if $a\in\mathbb{C}$, then the first main theorem of Nevanlinna theory becomes $m(r,a;f)+N(r,a;f)=T(r,f) + O(1)$, where $O(1)$ denotes a bounded function when $r$ is sufficiently large.
We define the order and the hyper-order of $f$ by
\[\rho(f):=\limsup _{r \rightarrow \infty} \frac{\log T(r, f)}{\log r}\;\text{and}\;\rho_1(f):=\limsup _{r \rightarrow \infty} \frac{\log \log T(r, f)}{\log r}.\]

\medskip
In recent years, the Nevanlinna value distribution theory in several complex variables has emerged as a prominent and rapidly growing area of research in complex analysis. This field has garnered significant attention due to its deep theoretical insights and wide-ranging applications in mathematics and related disciplines. Researchers (see \cite{ck}-\cite{CLL}, \cite{PVD2}, \cite{HZ1}, \cite{ps1}, \cite{BQL2}, \cite{BQL4}, \cite{FL}) have been particularly intrigued by its potential to extend classical results from one complex variable to higher-dimensional settings, as a result, this topic has become a focal point for contemporary studies in several complex variables.

\medskip
The first objective of this paper is to to estimate logarithmic partial derivative for meromorphic functions in several complex variables. In section 2, we extend the results of Gundersen \cite[Theorem 3, Corollary 2]{GG2} from $\mathbb{C}$ to $\mathbb{C}^m$. In section 3, we study Wiman-Valiron theory into higher dimensions. In section 4, we study the growth of solutions of a class of complex partial differential equations.

\section{\bf{Estimation for the logarithmic partial derivative for meromorphic function in several complex variables}}
It is well-known that the fundamental estimate for the logarithmic derivative (see \cite[pp. 55]{WKH2}), is expressed as
\[\displaystyle m\left(r, \frac{f^{(k)}}{f}\right)=S(r,f),\]
where $f$ denotes a non-constant meromorphic function in $\mathbb{C}$ and $k$ is a positive integer. This estimate finds wide applications in both the theory of meromorphic functions and algebraic differential equations. However, in certain problems concerning non-constant meromorphic functions of finite order, a more informative result can be derived by employing a different estimate. That's why instead of utilizing the aforementioned form, many authors considered an estimate of the form:
\bea\label{gd1} \displaystyle \left|\frac{f^{(k)}(z)}{f(z)}\right| \leqslant C|z|^\beta,\eea
where $z$ belongs to a specific set $G$. Here, $\beta$ and $C$ are finite real constants, and $k$ is a positive integer. For example, in prior works (such as \cite{GG1}, \cite{MO}) estimates of this form have been used to establish results related to solutions of linear differential equations of the form:
\[f^{(2)}(z) + A(z) f^{(1)}(z)+ B(z)f(z)=0,\]
where $A(z)$ and $B(z)$ are entire functions. In 1988, Gundersen \cite[Theorem 3, Corollary 2]{GG2} proved estimates that are both improvements and generalizations of these earlier estimates of the form (\ref{gd1}). For meromorphic functions of finite order, the estimates of Gundersen are best possible in a certain sense.

\smallskip
In 1988, Gundersen \cite{GG2} introduced the useful notation $(f, H)$, where $f$ is a transcendental meromorphic function in $\mathbb{C}$, and $H=\left\{(k_1, j_1), (k_2, j_2), \ldots, (k_q, j_q)\right\}$, a finite set of distinct integer pairs: The pairs satisfy the condition that $k_i>j_i\geq 0$ for $i=1,\ldots,q$. Additionally, for a non-negative integer $j$, we define $n_j(t)$ as the sum of the number of zeros and poles of $f^{(j)}$ (counting multiplicities) within the disk 
\[|z| \leq t: n_j(t)=n(t, f^{(j)}, 0) + n(t, f^{(j)}, \infty).\]

The results of Gundersen \cite{GG2} are as follows.
\begin{theo2A}\cite[Theorem 3]{GG2} Let $(f, H)$ be a given pair, and let $\alpha>1$ be a given real constant. Then there exists a set $E \subset (1, \infty)$ that has finite logarithmic measure, and there exist constants $A>0$ and $B>0$ that depend only on $\alpha$ and $H$, such that for all $z$ satisfying $|z| \not\in E \cup [0,1]$ and for all $(k,j)\in H$, we have the following two estimates (where $r = |z|$):
\beas\label{sn1} \displaystyle \bigg|\frac{f^{(k)}(z)}{f^{(j)}(z)}\bigg|&\leq& A\left(\frac{T(\alpha r,f)}{r}+\frac{n_j(\alpha r)}{r}\log^{\alpha} r\log^+n_j(\alpha r)\right)^{k-j}\eeas
and
\beas \displaystyle \bigg|\frac{f^{(k)}(z)}{f^{(j)}(z)}\bigg|&\leq& B\left(\frac{T(\alpha r,f)}{r}\log^{\alpha} r\log T(\alpha r,f)\right)^{k-j}.\eeas
\end{theo2A}

\begin{theo2B}\cite[Corollary 2]{GG2}
Let $(f, H)$ be a given pair where $f$ has finite order $\rho$, and let $\varepsilon>0$ be a given constant. Then there exists a set $E\subset (1,\infty)$ that has finite logarithmic measure, such that for all $z$ satisfying $|z|\not \in E \cup [0,1]$ and for all $(k,j)\in H$, we have
\[\displaystyle \bigg|\frac{f^{(k)}(z)}{f^{(j)}(z)}\bigg| \leq |z|^{(k-j)(\rho-1+\varepsilon)}.\]
\end{theo2B}

\smallskip
The estimation of logarithmic derivatives (see Theorems 2.A and 2.B) is a central tool in Nevanlinna theory and the value distribution of meromorphic functions in the complex plane. It connects growth properties of meromorphic (or entire) functions with the distribution of their zeros and poles. The estimation of logarithmic derivatives is a technical but fundamental tool-it shows up whenever one wants to link growth, uniqueness and value distribution of meromorphic functions or to study solutions of complex differential equations.

\smallskip
We define
\[\displaystyle \partial_{z_i}(f(z))=\frac{\partial f(z)}{\partial z_i},\ldots,\partial^{l_i}_{z_i}(f(z))=\frac{\partial^{l_i} f(z)}{\partial z_i^{l_i}},\]
where $l_i\in \mathbb{Z}^m_+$ and $i=1,2,\ldots,m$.

\smallskip
We now state our results, which estimate logarithmic partial derivative for meromorphic functions in several complex variables by extending the results of Gundersen \cite[Theorem 3, Corollary 2]{GG2} from $\mathbb{C}$ to $\mathbb{C}^m$.

\begin{theo}\label{t2.1} Let $f:\mathbb{C}^m\to \mathbb{P}^1$ be a transcendental meromorphic function, $I=(i_1,\ldots,i_m)\in \mathbb{Z}^m_+$ be a multi-index. Also let $I_n=(i_{n1},\ldots,i_{nm})$ such that $|I_n|=|I|+n$, where $n\geq 1$ is an integer. Define 
\[n_{\partial^{I}(f)}(r)=n_{\mu_{\partial^{I}(f)}^0}(r)+n_{\mu_{\partial^{I}(f)}^{\infty}}(r).\]

Then there exists a set $E \subset (1, \infty)$ of finite logarithmic measure such that
\beas \left|\frac{\partial^{I_n}(f(z))}{\partial^{I}(f(z))}\right|\leq B\left(\frac{T(\alpha^2 r,f)}{r}+\frac{n_{\partial^{I}(f)}(\alpha^2 r)}{r}\right)^{|I_n|-|I|},\eeas
for $z\in \mathbb{C}^m$ such that $||z||\not\in [0,1]\cup E$, where the constant $B$ depends only on $n$ and $|I|$.
\end{theo}

If $f$ is of finite order, then from Theorem \ref{t2.1}, we get the following corollary.
\begin{cor}\label{c2.1} Let $f:\mathbb{C}^m\to \mathbb{P}^1$ be a transcendental meromorphic function of finite order $\rho$ and let $I=(i_1,\ldots,i_m)\in \mathbb{Z}^m_+$ be a multi-index. Also let $I_n=(i_{n1},\ldots,i_{nm})$ such that $|I_n|=|I|+n$, where $n\geq 1$ is an integer. Then there exists a set $E \subset (1, \infty)$ of finite logarithmic measure such that
\beas \left|\frac{\partial^{I_n}(f(z))}{\partial^{I}(f(z))}\right|\leq ||z||^{(|I_n|-|I|)(\rho-1+\varepsilon)},\eeas
for all large values of $||z||$, where $z\in \mathbb{C}^m$ such that $||z||\not\in [0,1]\cup E$ and $\varepsilon>0$ is a given constant.
\end{cor}

\subsection {{\bf Auxiliary lemmas}}
In the proof of Theorem \ref{t2.1}, we use the following key lemmas.

\smallskip
Here we assume that $W$ is a hermitian vector space of dimension $m>1$. Define the holomorphic function $L:\mathbb{C}(1)\rightarrow \mathbb{C}$ by
\bea\label{sm1} L(z)=\frac{1}{(m-1)!}\frac{d^{m-1}}{dz^{m-1}}(z^{m-1}\log(1-z))=D_{m-1}\log(1-z),\eea
for $z\in \mathbb{C}(1)$, where $\log$ means the principal value of the logarithm.

Now we recall the following Jensen-Poisson Formula which was proved by W. Sto1l (see. \cite[Theorem 1.7]{WS4} or \cite[Theorem 5.2]{WS2}).
\begin{lem}\label{l2.1} Take $0<R\leq +\infty$. Let $f$ be a meromorphic function on $W(R)$ which is holomorphic at $0\in W(R)$ with $f(0)\neq 0$. Define $\nu=\mu_f$ and $A=\text{supp}\; \mu_f$. Assume $A\neq \varnothing$. Let $s$ be maximal with $0<s<R$ such that $A(s)=\varnothing$. Take $r\in \mathbb{R}[s,R)$. Then holomorphic functions $F$ and $E$ on $B(r)$ and defined by
\beas\label{sm2} \displaystyle F(\delta)&=&\int_{S(r)}\log |f(\eta)|\left[ \frac{r^{2m}}{(r^2-(\delta|\eta))^m}-1\right]\;\sigma(\eta),\\
\label{sm2.1} \displaystyle E(\delta)&=& \int_{A\cap B(r)}\nu(\eta)L\left(\frac{(\delta|\eta)}{r^2}\right)\;\omega^{m-1}(\eta),\eeas
for all $\delta\in B(r)$. Also a holomorphic function $H$ on $B(s)$ is defined by
\beas\label{sm3} \displaystyle H(\delta)=\int_{A\cap B(r)}\nu(\eta)L\left(\frac{(\delta|\eta)}{(\eta |\eta)}\right)\;\omega^{m-1}(\eta).\eeas
Moreover for $\delta\in B(s)$, we have
\[\displaystyle \log f(\delta)/f(0)=F(\delta)+H(\delta)-E(\delta).\]
\end{lem}

\begin{lem}\label{l2.2}
Let $f:\mathbb{C}^m\to \mathbb{P}^1$ be a transcendental meromorphic function and let $\alpha>1$ be a real constant. Suppose $g=\partial^{I}(f)$, where $I=(i_1,\ldots,i_m)\in \mathbb{Z}^m_+$ is a multi-index. Define $A=\text{supp}\; \mu_g$. Let $I_n=(i_{n1},\ldots,i_{nm})$ such that $|I_n|=|I|+n$, where $n\geq 1$ is an integer. Then
\beas \left|\frac{\partial^{I_n}(f(z))}{g(z)}\right|\leq B\left(\frac{T(\alpha^2 r,g)}{r}+\frac{n_{g}(\alpha^2 r)}{r}\right)^{|I_n|-|I|}\eeas
holds for all $z\in\mathbb{C}^m\backslash A$, where $||z||=r$, $n_{g}(r)=n_{\mu_{g}^0}(r)+n_{\mu_{g}^{\infty}}(r)$ and the constant $B$ depends only on $\alpha$, $n$ and $|I|$.
\end{lem}

\begin{proof} Define $\nu=\mu_{g}$. Assume that $A\neq \varnothing$.
Let $a=(a_1,a_2,\ldots,a_m)\in\mathbb{C}^m\backslash A$ such that $||a||=r>0$. For the sake of simplicity, we may assume that $g(a)=1$.
Let $s>0$ be maximal such that $\ol B(a,s)\cap A=\varnothing$. Suppose
\bea\label{eyy2} \displaystyle F(z)=\int_{S(a,\alpha s)}\log |g(\eta-a)|\;\left[\frac{(\alpha s)^{2m}}{((\alpha s)^2-(z-a|\eta-a))^m}-1\right]\;\sigma(\eta-a),\eea

\bea\label{eyy3} \displaystyle E(z)=\int_{A\cap B(a,\alpha s)}\nu(\eta)\;L\left(\frac{(z-a|\eta-a)}{(\alpha s)^2}\right)\;\omega^{m-1}(\eta-a)\eea
and
\bea\label{eyy4} \displaystyle H(z)=\int_{A\cap B(a,\alpha s)}\nu(\eta )\;L\left(\frac{(z-a|\eta-a)}{(\eta-a|\eta-a)}\right)\;\omega^{m-1}(\eta-a),\eea
where $z\in \ol B(a,s/2)$. Now by Lemma \ref{l2.1}, we have $\log g(z)=F(z)+H(z)-E(z)$ and so
\bea\label{eyy4a} \displaystyle \frac{\partial_{z_i}(g(z))}{g(z)}=\partial_{z_i}(F(z))+\partial_{z_i}(H(z))-\partial_{z_i}(E(z)),\eea
for all $z\in \ol B(a,s/2)$, where $i\in\mathbb{Z}[1,m]$. Let $J=(j_1,\ldots,j_m)\in\mathbb{Z}^m_+$ be a multi-index such that where $|J|=n-1$. Applying the operator $\partial^{J}$ on (\ref{eyy4a}), we get
\bea\label{eyy6} \partial^{J}\left(\frac{\partial_{z_i}(g(z))}{g(z)}\right)=\partial^{J}\big(\partial_{z_i}(F(z))\big)+\partial^{J}\big(\partial_{z_i}(H(z))\big)-\partial^{J}\big(\partial_{z_i}(E(z))\big),\eea
for all $z\in \ol B(a,s/2)$.

\medskip
Let $z=(z_1,z_2,\ldots, z_m)\in \ol B(a,s/2)$ and $\eta=(\eta_1,\eta_2,\ldots,\eta_m)\in S(a,\alpha s)$. By Schwarz inequality, we have $|(z-a|\eta-a)|\leq ||z-a||||\eta-a||\leq \alpha s^2/2$ and so
\bea\label{sm.1}\displaystyle |(\alpha s)^2-(z-a|\eta-a)|\geq (\alpha s)^2-|(z-a|\eta-a)|\geq\alpha (\alpha-\frac{1}{2})s^2.\eea

\medskip
It is easy to verify that
\bea\label{sm.2} \displaystyle |\ol {\eta_i-a_i}|\leq \sqrt{|\eta_1-a_1|^2+\ldots+|\eta_m-a_m|^2}<\alpha s.\eea

\medskip
Using (\ref{sm.1}) and (\ref{sm.2}), we get
\bea\label{sm.3}
\displaystyle \left| \frac{(\alpha s)^{2m}}{((\alpha s)^2-(z-a|\eta-a))^{m+n}}\;\ol {\eta_1-a_1}^{j_1}\cdots \ol { \eta_m-a_m}^{j_m}\ol {\eta_i-a_i}\right|\leq \frac{2^{m+n}\alpha^m}{(2\alpha-1)^{m+n}s^n},\eea
for all $z=(z_1,z_2,\ldots, z_m)\in \ol B(a,s/2)$ and $\eta=(\eta_1,\eta_2,\ldots,\eta_m)\in S(a,\alpha s)$.
Now (\ref{eyy2}) yields
\bea\label{sm.4}&&\displaystyle\partial^{J}\left(\partial_{z_i}(F(z))\right)\\&=&\frac{(m+n-1)!}{(m-1)!}\times\nonumber\\&&\int_{S(a,\alpha s)}\log |g(\eta-a)|\frac{(\alpha s)^{2m}}{((\alpha s)^2-(z-a|\eta-a))^{m+n}}\ol {\eta_1-a_1}^{j_1}\cdots \ol {\eta_m-a_m}^{j_m}\ol {\eta_i-a_i}\sigma(\eta-a),\nonumber\eea
for $z\in \ol B(a,s/2)$. Therefore using (\ref{sm.3}) to (\ref{sm.4}), we get
\bea\label{sm.5}\displaystyle \left|\partial^{J}\left(\partial_{z_i}(F(z))\right)\right|&\leq& \frac{(m+n-1)!2^{m+n}\alpha^m}{(m-1)!(2\alpha-1)^{m+n}s^n}\int_{S(a,\alpha s)}
\left|\log |g(\eta-a)|\right|\;\sigma(\eta-a)\\&=&
\frac{(m+n-1)!2^{m+n}\alpha^m}{(m-1)!(2\alpha-1)^{m+n}s^n}\int_{S(\alpha s)}\left|\log |g(\zeta)|\right|\;\sigma(\zeta)\nonumber\\&=&
\frac{(m+n-1)!2^{m+n}\alpha^m}{(m-1)!(2\alpha-1)^{m+n}s^n}\int_{S(\alpha s)}\left(\log^+ |g(\zeta)|+\log^+\frac{1}{|g(\zeta)|}\right)\;\sigma(\zeta)\nonumber\\&=&
\frac{(m+n-1)!2^{m+n}\alpha^m}{(m-1)!(2\alpha-1)^{m+n}s^n}\left(m\left(\alpha s,g\right)+m\left(\alpha s, \frac{1}{g}\right)\right)\nonumber\\&\leq&
\frac{(m+n-1)!2^{m+n}\alpha^m}{(m-1)!(2\alpha-1)^{m+n}s^n}\left(T\left(\alpha s,g\right)+T\left(\alpha s, \frac{1}{g}\right)\right),\nonumber\eea
for all $z\in \ol B(a,s/2)$. Now from (\ref{la}) and (\ref{sm.5}), we get
\bea\label{sm.6}\displaystyle \left|\partial^{J}\left(\partial_{z_i}(F(z))\right)\right|\leq \frac{2^{m+n+1}(m+n-1)!\alpha^m}{(m-1)!(2\alpha-1)^{m+n}s^n}T\left(\alpha s,g\right),\eea
for all $z\in \ol B(a,s/2)$.

\medskip
Let
\[h(w)=w^{m-1}\log (1-w),\]
where $w=\frac{(z-a|\eta-a)}{(\alpha s)^2}$. If $z=(z_1,z_2,\ldots, z_m)\in \ol B(a,s/2)$ and $\eta\in A\cap B(a,\alpha s)$, then
\[\displaystyle |(z-a|\eta-a)|\leq ||z-a||||\eta-a||\leq \alpha s^2/2\]
and so $|w|\leq\frac{1}{2\alpha}<\frac{1}{2}$. Clearly $h(w)$ is holomorphic in $|w|< \frac{1}{2}$ and so
\beas \displaystyle |h(w)|=|w|^{m-1}\Big|w+\frac{w^2}{2}+\frac{w^3}{3}+\ldots\Big|
<\frac{1}{2^m}\left(1+\frac{1}{2}+\frac{1}{2^2}+\ldots\right)=\frac{1}{2^{m-1}}.\eeas

\medskip
It follows from Cauchy Estimates that
\bea\label{sm.7} \left|\frac{d^n}{dw^n}L\left(\frac{(z-a|\eta-a)}{(\alpha s)^2}\right)\right|=\frac{1}{(m-1)!}\left|\frac{d^{m+n-1}}{dw^{m+n-1}}h(w)\right|
\nonumber&\leq& \frac{(m+n-1)!\sup\limits_{|w|\leq \frac{1}{2}}|h(w)|}{(m-1)!\frac{1}{2^{m+n-1}}}\\&=&\frac{(m+n-1)!2^n}{(m-1)!}.\eea

\medskip
From (\ref{sm.2}), we have
\bea\label{sm.8} \displaystyle \left|\frac{\ol {\eta_1-a_1}^{j_1}\cdots \ol {\eta_m-a_m}^{j_m}\ol {\eta_i-a_i}}{(\alpha s)^{2n}}\right|\leq \frac{1}{(\alpha s)^n}.\eea

\medskip
Now from (\ref{eyy3}), we get
\bea\label{sm.9}\displaystyle\partial^{J}\left(\partial_{z_i}(E(z))\right)=\int_{A\cap B(a,\alpha s)}\nu(\eta) \frac{d^n L(w)}{dw^n}\frac{\ol {\eta_1-a_1}^{j_1}\cdots \ol {\eta_m-a_m}^{j_m}\ol {\eta_i-a_j}}{(\alpha s)^{2n}}\omega^{m-1}(\eta-a)\eea
for all $z\in \ol B(a,s/2)$. Therefore using (\ref{sm.7}) and (\ref{sm.8}) to (\ref{sm.9}), we get
\bea\label{sm.10}\displaystyle \left|\partial^{J}\left(\partial_{z_i}(E(z))\right)\right|&\leq& \frac{(m+n-1)!2^n}{(m-1)!(\alpha s)^n}
\int_{A\cap B(a,\alpha s)}|\nu(\eta)|\;\omega^{m-1}(\eta-a)
\\&\leq& \frac{(m+n-1)!2^n}{(m-1)!(\alpha s)^n}\int_{A\cap B(\alpha s)}|\nu(\zeta+a)|\;\omega^{m-1}(\zeta)\nonumber
,\eea
for all $z\in \ol B(a,s/2)$. Note that
\[\displaystyle |\nu(\zeta+a)|\leq \mu_g^0(\zeta+a)+\mu_g^{\infty}(\zeta+a).\]

\smallskip
Then from (\ref{sm.10}), we get
\bea\label{sm.11}\displaystyle \left|\partial^{J}\left(\partial_{z_i}(E(z))\right)\right|&\leq& \frac{(m+n-1)!2^n}{(m-1)!(\alpha s)^n}
\int_{A\cap B(\alpha s)}\mu_g^0(\zeta+a)\;\omega^{m-1}(\zeta)\\&&+\frac{(m+n-1)!2^n}{(m-1)!(\alpha s)^n}
\int_{A\cap B(\alpha s)}\mu_g^{\infty}(\zeta+a)\;\omega^{m-1}(\zeta),\nonumber\eea
for all $z\in \ol B(a,s/2)$. We know that (see \cite[pp. 10]{WS1})
\bea\label{sm.12} \displaystyle n_{\mu_g^c}(\alpha s)=\int_{A\cap B(\alpha s)}\mu_g^c(\zeta+a)\;\omega^{m-1}(\zeta)+\mu_g^c(a),\eea
where $c=0$ or $\infty$. Since $g(a)=1$, using (\ref{sm.12}) to (\ref{sm.11}), we get
\bea\label{eyy16}\displaystyle \left|\partial^{J}\left(\partial_{z_i}(E(z))\right)\right|\leq \frac{(m+n-1)!2^n}{(m-1)!(\alpha s)^n}
n_{g}(\alpha s),\eea
for all $z\in \ol B(a,s/2)$.

\medskip
Let
\[\displaystyle h(\tilde w)=\tilde w^{m-1}\log (1-\tilde w),\]
where $\tilde w=\frac{(z-a|\eta-a)}{(\eta-a|\eta-a)}$. Suppose $\eta\in A\cap B(a,\alpha s)$. Since $\alpha>1$ and $A\cap B(a,s)=\varnothing$, we have $||\eta-a||\geq s$. Then by Schwarz inequality, we get $|\tilde w|\leq \frac{1}{2}$ and so $h(\tilde w)$ is holomorphic in $|\tilde w|\leq \frac{1}{2}$. Clearly
$|h(\tilde w)|\leq \frac{1}{2^{m-1}}.$
It follows from Cauchy Estimates that
\bea\label{eyy17} \displaystyle\left|\frac{d^n}{d\tilde w^n}L\left(\frac{(z-a|\eta-a)}{(\eta-a|\eta-a)}\right)\right|=\frac{1}{(m-1)!}\left|\frac{d^{m+n-1}}{d\tilde w^{m+n-1}}h(\tilde w)\right|\leq \frac{(m+n-1)!2^n}{(m-1)!}.\eea

\medskip
Also it is easy to verify that
\bea\label{eyy18}\displaystyle\left|\frac{\ol {\eta_1-a_1}^{j_1}\cdots \ol {\eta_m-a_m}^{j_m}\ol {\eta_i-a_i}}{||\eta-a||^{2n}}\right|\leq\frac{||\eta-a||^n}{||\eta-a||^{2n}}\leq \frac{1}{s^n}\eea
for all $z\in \ol B(a,s/2)$. On the other hand from (\ref{eyy4}), we get
\bea\label{eyy19}\displaystyle\partial^{J}\left(\partial_{z_i}(H(z))\right)=\int_{A\cap B(a,\alpha s)}\nu(\eta) \frac{d^n L(\tilde w)}{d\tilde w^n}\frac{\ol {\eta_1-a_1}^{j_1}\cdots \ol {\eta_m-a_m}^{j_m}\ol {\eta_i-a_i}}{||\eta-a||^{2n}}\omega^{m-1}(\eta)\eea
for all $z\in \ol B(a,s/2)$. Therefore using (\ref{eyy17}) and (\ref{eyy18}) to (\ref{eyy19}), we get
\beas\displaystyle \left|\partial^{J}\left(\partial_{z_i}(H(z))\right)\right|&\leq& \frac{(m+n-1)!2^n}{(m-1)!s^n}
\int_{A\cap B(a,\alpha s)}|\nu(\eta)|\;\omega^{m-1}(\eta-a)\\&\leq&
\frac{(m+n-1)!2^n}{(m-1)!s^n}\int_{A\cap B(\alpha s)}\mu_g^0(\zeta+a)\;\omega^{m-1}(\zeta)\\&&
+\frac{(m+n-1)!2^n}{(m-1)!s^n}
\int_{A\cap B(\alpha s)}\mu_g^{\infty}(\zeta+a)\;\omega^{m-1}(\zeta)
,\eeas
for all $z\in \ol B(a,s/2)$ and so from (\ref{sm.12}), we get
\bea\label{eyy20}\displaystyle \left|\partial^{J}\left(\partial_{z_i}(H(z))\right)\right|\leq \frac{(m+n-1)!2^n}{(m-1)!s^n}n_{g}(\alpha s),\eea
for all $z\in \ol B(a,s/2)$.

\medskip
Now (\ref{sm.6}), (\ref{eyy16}) and (\ref{eyy20}) into (\ref{eyy6}), we get
\bea\label{eyy21a} \left|\partial^{J}\left(\frac{\partial_{z_i}(g(z))}{g(z)}\right)\right|&\leq& \frac{2^{m+n+1}(m+n-1)!\alpha^m}{(m-1)!(2\alpha-1)^{m+n}s^n}T(\alpha s,g)+\frac{(m+n-1)!2^n(\alpha^n+1)}{(m-1)!\alpha^n s^n}n_{g}(\alpha s)\nonumber\\&\leq&
\frac{2^{m+n+1}(m+n-1)!\alpha^{m}}{(m-1)!(2\alpha-1)^{m+n}}\left(\frac{T(\alpha s,g)}{s^n}+\frac{n_{g}(\alpha s)}{s^n}\right)
\eea
for all $z\in \ol B(a,s/2)$.

Let us choose $\theta>0$ such that $s=\theta r$. If $\theta\neq 1$,  we take $\alpha>\max\{\theta,1/\theta\}$. Clearly $\alpha>1$ and $1/\theta^n<\alpha^n$. Then from (\ref{eyy21a}), we get
\beas \left|\partial^{J}\left(\frac{\partial_{z_i}(g(a))}{g(a)}\right)\right|&\leq&
\frac{2^{m+n+1}(m+n-1)!\alpha^{m}}{(m-1)!\theta^n (2\alpha-1)^{m+n}}\left(\frac{T(\alpha^2 r,g)}{r^n}+\frac{n_{g}(\alpha^2 r)}{r^n}\right)\\&\leq&
\frac{2^{m+n+1}(m+n-1)!\alpha^{m}}{(m-1)!\theta^n \alpha^{m+n}(2-1/\alpha)^{m+n}}\left(\frac{T(\alpha^2 r,g)}{r^n}+\frac{n_{g}(\alpha^2 r)}{r^n}\right)\\&\leq&
\frac{2^{m+n+1}(m+n-1)!}{(m-1)!}\left(\frac{T(\alpha^2 r,g)}{r^n}+\frac{n_{g}(\alpha^2 r)}{r^n}\right).
\eeas

If $\theta=1$, the above inequality also holds. Since $a\in\mathbb{C}^m\backslash A$ was arbitrary, it follows that
\bea\label{eyy21} \left|\partial^{J}\left(\frac{\partial_{z_i}(g(z))}{g(z)}\right)\right|\leq
\frac{2^{m+n+1}(m+n-1)!}{(m-1)!}\left(\frac{T(\alpha^2 r,g)}{r^n}+\frac{n_{g}(\alpha^2 r)}{r^n}\right),
\eea
for all $z\in\mathbb{C}^m\backslash A$.

\medskip
If $|J|=0$, then from (\ref{eyy21}), we get
\bea\label{eyy22} \left|\frac{\partial_{z_i}(g(z))}{g(z)}\right|\leq A_{i_1}\left(\frac{T(\alpha^2 r,g)}{r}+\frac{n_{g}(\alpha^2 r)}{r}\right)
,\eea
for all $z\in\mathbb{C}^m\backslash A$ and any $i\in\mathbb{Z}[1,m]$, where $A_{i_1}>0$ is a constant that depends only on $n$.

\smallskip
Let $I_1=(i_{11},\ldots,i_{1i},\ldots,i_{1m})=(i_1,\ldots, i_i+1,\ldots,i_m)$. Clearly $|I_1|=|I|+1$ and so from (\ref{eyy22}), we have
\beas \left|\frac{\partial^{I_1}(f(z))}{\partial^I(f(z))}\right|\leq A_{I_1}\left(\frac{T(\alpha^2 r,g)}{r}+\frac{n_{g}(\alpha^2 r)}{r}\right)^{|I_1|-|I|},\eeas
for all $z\in\mathbb{C}^m\backslash A$, where $A_{i_1}>0$ is a constant.

\smallskip
Suppose $|J|=1$. Let $\partial^J(g)=\partial_{z_j}(g)$ for any $j\in\mathbb{Z}[1,m]$. For the sake of simplicity, we assume that $i\leq j$. Let $I_2=(i_{21},\ldots,i_{2i},\ldots, i_{2j},\ldots,i_{2m})=(i_1,\ldots, i_i+1,\ldots,i_j+1,\ldots,i_m)$. Clearly $|I_2|=|I|+2$ and $\partial^{|I_2|}(f(z))=\partial_{z_j}(\partial_{z_i}(g(z)))$.
Note that
\bea\label{eyy23} \partial_{z_j}\left(\frac{\partial_{z_i}(g)}{g}\right)=\frac{\partial_{z_j}(\partial_{z_i}(g))}{g}-\frac{\partial_{z_j}(g)}{g}\frac{\partial_{z_i}(g)}{g}.\eea

\smallskip
Now from (\ref{eyy21})-(\ref{eyy23}), we deduce that
\bea\label{eyy24} \left|\frac{\partial^{I_2}(f(z))}{\partial^I(f(z))}\right|=\left|\frac{\partial_{z_j}(\partial_{z_i}(g(z)))}{g(z)}\right|&\leq&
\left|\frac{\partial_{z_j}(g(z))}{g(z)}\right|\times \left|\frac{\partial_{z_i}(g(z))}{g(z)}\right|+\left|\partial_{z_j}\left(\frac{\partial_{z_i}(g(z))}{g(z)}\right)\right|\nonumber\\
&\leq& A_{I_2}\left(\frac{T(\alpha^2 r,g)}{r}+\frac{n_{g}(\alpha^2 r)}{r}\right)^{|I_2|-|I|},\eea
for all $z\in\mathbb{C}^m\backslash A$, where $A_{I_2}>0$ is a constant.

Suppose $|J|=2$. Let $\partial^J(g)=\partial_{z_k}(\partial_{z_l}(g))$ for any $k,l\in\mathbb{Z}[1,m]$. Assume that $i\leq k\leq l$. Let
$I_3=(i_{31},\ldots,i_{3i},\ldots,i_{3k},\ldots,i_{3l},\ldots,i_{3m})=(i_1,\ldots, i_i+1,\ldots,i_k+1,\ldots,i_l+1,\ldots,i_m).$
Clearly $|I_3|=|I|+3$ and $\partial^{|I_3|}(f(z))=\partial_{z_l}(\partial_{z_k}(\partial_{z_i}(g(z))))$.
Also we see that
\bea\label{eyy25} \partial_{z_l}\left(\partial_{z_k}\left(\frac{\partial_{z_i}(g)}{g}\right)\right)&=&\frac{\partial_{z_l}(\partial_{z_k}(\partial_{z_i}(g)))}{g}-\frac{\partial_{z_l}(g)}{g}\frac{\partial_{z_k}(\partial_{z_i}(g))}{g}-\frac{\partial_{z_i}(g)}{g}\frac{\partial_{z_l}(\partial_{z_k}(g))}{g}\nonumber\\&&-\frac{\partial_{z_k}(g)}{g}\frac{\partial_{z_l}(\partial_{z_i}(g))}{g}+2\frac{\partial_{z_i}(g)}{g}\frac{\partial_{z_k}(g)}{g}\frac{\partial_{z_l}(g)}{g}.\eea

\smallskip
Now using (\ref{eyy21})-(\ref{eyy25}), we deduce that
\beas \left|\frac{\partial^{I_3}(f(z))}{\partial^I(f(z))}\right|&=&\left|\frac{\partial_{z_l}(\partial_{z_k}(\partial_{z_i}(g(z))))}{g(z)}\right|\\&\leq&
2\left|\frac{\partial_{z_i}(g(z))}{g(z)}\right|\times \left|\frac{\partial_{z_k}(g(z))}{g(z)}\right|\times \left|\frac{\partial_{z_l}(g(z))}{g(z)}\right|+\left|\frac{\partial_{z_l}(g(z))}{g(z)}\frac{\partial_{z_k}(\partial_{z_i}(g))}{g}\right|\nonumber\\&&+
\left|\frac{\partial_{z_i}(g)}{g}\frac{\partial_{z_l}(\partial_{z_k}(g(z)))}{g(z)}\right|+\left|\frac{\partial_{z_k}(g)}{g}\frac{\partial_{z_l}(\partial_{z_i}(g(z)))}{g(z)}\right|\nonumber\\&&+
\left|\partial_{z_l}\left(\partial_{z_k}\left(\frac{\partial_{z_i}(g(z))}{g(z)}\right)\right)\right|
\leq A_{I_3}\left(\frac{T(\alpha^2 r,g)}{r}+\frac{n_{g}(\alpha^2 r)}{r}\right)^{|I_3|-|I|},\nonumber\eeas
for all $z\in\mathbb{C}^m\backslash A$, where $A_{I_3}>0$ is a constant.
By repeating this process, it can be deduced from finite induction that for each positive integer $q$, we have
\bea\label{eyy26} \left|\frac{\partial^{I_q}(f(z))}{\partial^I(f(z))}\right|\leq A_{I_q}\left(\frac{T(\alpha^2 r,g)}{r}+\frac{n_{g}(\alpha^2 r)}{r}\right)^{|I_q|-|I|},\eea
for all $z\in\mathbb{C}^m\backslash A$, where $I_q=(i_{q1},\ldots,i_{qm})$ such that $|I_q|=|I|+q$ and $A_{I_q}>0$ is a constant that depends only on $q$. Then from (\ref{eyy26}), we get
\beas\label{ey1} \displaystyle \left|\frac{\partial^{I_n}(f(z))}{\partial^{I}(f(z))}\right|\leq B\left(\frac{T(\alpha^2 r,g)}{r}+\frac{n_{g}(\alpha^2 r)}{r}\right)^{|I_n|-|I|},\eeas
holds for all $z\in\mathbb{C}^m\backslash A$, where $||z||=r$ and the constant $B$ depends on $n$.
\end{proof}

Next we recall the lemma of logarithmic derivative:

\begin{lem}\label{l2.3}\cite[Lemma 1.37]{HLY1} \cite{ZY1} Let $f:\mathbb{C}^m\to\mathbb{P}^1$ be a non-constant meromorphic function and let $I=(i_1,\ldots,i_m)\in \mathbb{Z}^m_+$ be a multi-index. Then for any $\varepsilon>0$, we have
\[m\left(r,\frac{\partial^I(f)}{f}\right)\leq |I|\log^+T(r,f)+|I|(1+\varepsilon)\log^+\log T(r,f)+O(1)\]
for all large $r$ outside a set $E$ with $\text{log mes}\;E=\int_E d\log r<\infty$.
\end{lem}

\begin{lem}\label{l2.4} Let $f:\mathbb{C}^m\to\mathbb{P}^1$ be a non-constant meromorphic function and $I=(\alpha_1,\ldots,\alpha_n)\in \mathbb{Z}^m_+$ be a multi-index. Let $\alpha>1$ be a real constant. If $a\in\mathbb{P}^1$, then we have
\bea\label{exx0} n_{\mu^a_{\partial^{I}(f)}}(r)\leq \frac{|I|+3}{\log \alpha}T(\alpha r,f),\eea
for all large $r$ outside a set $E_1$ with $\text{log mes}\;E_1<+\infty$.
\end{lem}

\begin{proof} Let us assume that $l=\mu_f^{\infty}(z_0)>0$ for some $z_0\in\mathbb{C}^m$. Note that we can find holomorphic functions $g$ and $h$ such that $\dim g^{-1}(\{0\})\cap \dim h^{-1}(\{0\})\leq m-2$ and $f=\frac{g}{h}$. Set
\[S=\big(g^{-1}(\{0\})\cap h^{-1}(\{0\})\big)\cup h^{-1}(\{0\})_s,\]
where $h^{-1}(\{0\})_s$ is the set of singular points of $h^{-1}(\{0\})$. Then $S$ is an analytic subset of $\mathbb{C}^m$ of dimension $\leq m-2$. With out loss of generality, we assume that $z_0\in h^{-1}(\{0\})-S$ is a regular point. Then there is a holomorphic coordinate system $\left(U ; \varphi_1, \ldots, \varphi_m\right)$ of $z_0$ in $\mathbb{C}^m-g^{-1}(\{0\})$ such that 
\[U \cap h^{-1}(\{0\})=\left\{z \in U \mid \varphi_1(z)=0\right\}\;\text{and}\;\;\varphi_j(z_0)=0\]
for $j=1,2,\ldots,m$ (see proof of Lemma 2.3 \cite{FL1}). A biholomorphic coordinate transformation $z_j=z_j(\varphi_1, \ldots, \varphi_m)$, $j=1, \ldots, m$ near $0$ exists such that $z_0=(z_1(0), \ldots, z_m(0))$ and $f=\varphi_1^{-l} \hat{f}\left(\varphi_1, \ldots, \varphi_m\right)$,
where $\hat{f}$ is a holomorphic function near $0$ and zero free along the set $h^{-1}(\{0\})$. Now for any $i\in\mathbb{Z}[1, m]$, we have
\bea\label{sbb.3a} \partial_i(f)=\sideset{}{_{j=1}^m}{\sum} \frac{\partial f}{\partial \varphi_j}\frac{\partial \varphi_j}{\partial z_i}=\frac{-l}{\varphi_1^{l+1}}\hat{f}\frac{\partial \varphi_1}{\partial z_i}+\frac{1}{\varphi_1^{l}}\sideset{}{_{j=1}^m}{\sum}\frac{\partial \hat{f}}{\partial \varphi_j}\frac{\partial \varphi_j}{\partial z_i},\eea
which means
$\mu^{\infty}_{\partial_{z_i}(f)}(z_0)\leq \mu^{\infty}_f(z_0)+\mu^{\infty}_{f,1}(z_0)$.
By induction, we get
\[\mu^{\infty}_{\partial^I(f)}(z_0)\leq \mu^{\infty}_f(z_0)+|I|\mu^{\infty}_{f,1}(z_0)\]
and so
\[N(r,\partial^{I}(f))\leq N(r,f)+|I|\ol N(r,f).\]

Using Lemma \ref{l2.3}, we get
\bea\label{exx2}T(r,\partial^{I}(f))=N(r,\partial^{I}(f))+m(r,\partial^{I}(f))&\leq& N(r,\partial^{I}(f))+m(r,f)+o(T(r,f))
\\&\leq&(1+o(1))(|I|+1)T(r,f),\nonumber\eea
for all $r$ outside a set $E_1$ with $\text{log mes}\;E_1<+\infty$.
Note that
$n_{\mu^a_f}(r)=r\frac{\partial}{\partial r} N_{\mu^a_f}(r)$. For $\alpha>1$, we have
\bea\label{exx4} n_{\mu^a_{\partial^{I}(f)}}(r)\log \alpha\leq \int_{r}^{\alpha r} \frac{n_{\mu^a_{\partial^{I}(f)}}(t)}{t}dt\leq N_{\mu^a_{\partial^{I}(f)}}(\alpha r).\eea

Since $N_{\mu^a_{\partial^{I}(f)}}(\alpha r)\leq T(\alpha r,\partial^{I}(f))+O(1)$ as $r\rightarrow \infty$, from (\ref{exx4}), we get
\bea\label{exx5} n_{\mu^a_{\partial^{I}(f)}}(r)\leq (1+o(1))\frac{T(\alpha r,\partial^{I}(f))}{\log \alpha}\;\;\text{as}\;\; r\rightarrow \infty.\eea

If $|I|=0$, then (\ref{exx0}) follows from (\ref{exx5}). If $|I|>0$, then by (\ref{exx2}) and (\ref{exx5}), we have
\[n_{\mu^a_{\partial^{I}(f)}}(r)\leq \frac{|I|+3}{\log \alpha}T(\alpha r,f),\]
for all large $r$ outside a set $E_1$ with $\text{log mes}\;E_1<+\infty$.
\end{proof}

The following lemma is the Hinkkanen's Borel-type growth lemma (or see similar lemmas \cite[Lemma 3.3.1]{CY1}) and \cite[Lemma 7]{ZY1}).
\begin{lem}\cite[Lemma 4]{AH}\label{l3.1} Let $\varphi(r)$ and $\psi(r)$ be positive nondecreasing functions defined for $r\geq r_1>0$ and $r\geq r_2>0$ respectively such that
\beas \int_{r_1}^{\infty} \frac{d r}{\varphi(r)}=\infty\;\;\text{and}\;\; \quad \int_{r_2}^{\infty} \frac{d r}{r \psi(r)}<\infty.\eeas

Let $T(r)$ be a positive nondecreasing function defined for $r\geq r_3 \geq r_1$ and $T(r)\geq r_2$. Then if $C_1$ is real with $C_1>1$, one has
\[T\left(r+\frac{\varphi(r)}{\psi(T(r))}\right) \leq C_1 T(r)\]
whenever $r\geq r_3$ and $r\notin E_1$, where $\int_{E_1} \frac{d r}{\varphi(r)}<+\infty$.
\end{lem}

\medskip
Let $U$ be an open subset of $\mathbb{C}^m$. Let $A \subset  U$ be an analytic subset of pure dimension $l$ and let $\iota_{R(A)} : R(A)\longrightarrow \mathbb{C}^m$ be the inclusion map. For a compact subset $K\subset U$ the integral
\beas \displaystyle \int_{K\cap A} \alpha ^l=\int_{K\cap R(A)} \alpha^l = \int_{K\cap R(A)} {\iota}^{\ast} _{R(A)} \alpha^l  \eeas
is considered as a measure of $K\cap A$. We now recall the following lemma.

\begin{lem}\cite[Lemma 2.2.9]{NW}\label{l2.5} Let the notation be as above.
\begin{enumerate}
\item[(i)] If $l=\dim A < m$, the Lebesgue measure of $A$ in $\mathbb{C}^m$ is zero.
\item [(ii)] $\int_{K\cap R(A)} \alpha^l<\infty$.
\end{enumerate}
\end{lem}

\medskip
Let $A \subset \mathbb{C}^m$ be an analytic subset of pure dimension $l<m$. Consider the collection $\mathscr{K}$ of all compact sets in $\mathbb{C}^m$. Then
\[\displaystyle m^*_m(A)=\inf\left\lbrace \sum\limits_{j=1}^{\infty} \int_{K_j\cap A} \alpha ^l: K_1,K_2,\ldots\in \mathscr{K}\;\;\text{so that}\;\; A\subset \bigcup\limits_{j=1}^{\infty} K_j \right\rbrace,\]
which is called the Lebesgue outer measure of $A$ in $\mathbb{C}^m$. Clearly by Lemma \ref{l2.5}, we have $m^*_m(A)=0$. Also by the definition of $m^*_m$,
for a given $\delta>0$ there exists a countable collection $\{K_j\}$ of compact sets such that $A\subset \bigcup\limits_{j=1}^{\infty} K_j$ satisfying
\bea\label{ee}\displaystyle \sum\limits_{j=1}^{\infty} \int_{K_j\cap A} \alpha ^l<\delta.\eea

Since $A$ is a closed set, it follows that $K_j\cap A$ is also a compact set for $j=1,2,\ldots$ and so the set $K_j\cap A$ is also measurable for $j=1,2,\ldots$. Again sine the integral
$\int_{K_j\cap A} \alpha ^l$ is considered as a measure of $K\cap A$, from (\ref{ee}), we have
\bea\label{eee} \displaystyle \sum\limits_{j=1}^{\infty} m(K_j\cap A)<\delta.\eea

We know that the function $\tau:\mathbb{C}^m\to \mathbb{R}[0,\infty)$ defined by $\tau(z)=||z||$ is continuous. Consequently $\tau(K_j\cap A)$ is also a compact set in $\mathbb{R}[0,\infty)$ for $j=1,2,\ldots$. Since
\[\displaystyle A=\bigcup\limits_{j=1}^{\infty}(K_j\cap A),\]
it follows that
\[\displaystyle \tau(A)\subset\bigcup\limits_{j=1}^{\infty} \tau(K_j\cap A)\]
and so (\ref{eee}) gives
\bea\label{eeee} \displaystyle m^*_m(\tau(A))\leq \bigcup\limits_{j=1}^{\infty} m^*_m(\tau(K_j\cap A))=\bigcup\limits_{j=1}^{\infty} m_m(\tau(K_j\cap A))<\delta.\eea

We know that $\tau$ is continuous and $A$ is closed. Therefore $\tau(A)$ is also closed and so $\tau(A)$ is measurable. Consequently from (\ref{eeee}), we have $m_m(\tau(A))<\delta$ and so the set $\tau(A)$ is of finite measurable.
Let $E_1=\tau(A)\backslash  [0,1]$. Then $m_m(E_1)<\delta$. We know that $E$ is of finite logarithmic measure whenever it is of finite linear
measure and so trivially we have
\bea\label{q1}\displaystyle \text{log mes}\;E_1\leq m_m(E_1)<\delta<+\infty.\eea

\subsection {{\bf Proof of Theorem \ref{t2.1}}}
\begin{proof} Define $A=\text{supp}\; \mu_{\partial^{I}(f)}$. Clearly $A$ is an analytic set of pure dimension $m-1$ and $m_m(A)=0$.
Let $\alpha>1$ be a real constant. Then from Lemma \ref{l2.2}, one can easily deduce that
\bea\label{ey1} \left|\frac{\partial^{I_n}(f(z))}{\partial^{I}(f(z))}\right|\leq B_1\left(\frac{T(\alpha^2 r,\partial^{I}(f))}{r}+\frac{n_{\partial^{I}(f)}(\alpha^2 r)}{r}\right)^{|I_n|-|I|},\eea
for $z\in \mathbb{C}^m\backslash A$, where the constant $B_1$ depends only on $n$ and $|I|$.
Let $E_1=\tau(A)\backslash  [0,1]$. Now from (\ref{q1}), we see that $\text{log mes}\;E_1<+\infty$.
Therefore from (\ref{ey1}), we get
\bea\label{ey2} \left|\frac{\partial^{I_n}(f(z))}{\partial^{I}(f(z))}\right|\leq B\left(\frac{T(\alpha^2 r,\partial^{I}(f))}{r}+\frac{n_{\partial^{I}(f)}(\alpha^2 r)}{r}\right)^{|I_n|-|I|},\eea
for $z\in \mathbb{C}^m$ such that $||z||\not\in [0,1]\cup E_1$. Also from (\ref{exx2}), we have
\[T(\alpha^2 r,\partial^{I}(f))\leq 2(|I|+1)T(2\alpha^2r,f)\]
for all large $r$ outside a set $E_2$ with $\text{log mes}\;E_2<+\infty$. Let $E=E_1\cup E_2$. Clearly $\text{log mes}\;E<+\infty$.
Finally from (\ref{ey2}), we get
\bea\label{ey4} \left|\frac{\partial^{I_n}(f(z))}{\partial^{I}(f(z))}\right|\leq B\left(\frac{T(\alpha^2 r,f)}{r}+\frac{n_{\partial^{I}(f)}(\alpha^2 r)}{r}\right)^{|I_n|-|I|},\eea
for $z\in \mathbb{C}^m$ such that $||z||\not\in [0,1]\cup E$, where the constant $B$ depends only on $n$ and $|I|$.
\end{proof}

\subsection {{\bf Proof of Corollary \ref{c2.1}}}
\begin{proof} At first from (\ref{exx0}), we have
\bea\label{ey5} n_{\partial^{I}(f)}(\alpha^2 r)\leq 2\frac{|I|+3}{\log \alpha}T(\alpha^3 r,f),\eea
for all large $r$ outside a set $E_3$ with $\text{log mes}\;E_3<+\infty$. Then from (\ref{ey4}) and (\ref{ey5}), we get
\bea\label{ey6} \left|\frac{\partial^{I_n}(f(z))}{\partial^{I}(f(z))}\right|\leq C\left(\frac{T(\alpha^3 r,f)}{r}\right)^{|I_n|-|I|},\eea
for $z\in \mathbb{C}^m$ such that $||z||\not\in [0,1]\cup E$, where $E=E_1\cup E_2\cup E_3$ and the constant $C>0$ depends only on $n$ and $|I|$.
Let us choose
\[\alpha^3=1+\frac{\varphi(r)}{r\psi(T(r,f))},\]
where $\psi(r)=(\log r)^{1+\varepsilon_1}\;(\varepsilon_1>0)$ and $\varphi(r)=r$. Now applying Lemma \ref{l3.1} to the increasing function $T(r,f)$, we have
\bea\label{ey6a}T(\alpha^3 r,f)=T\left(r+\frac{\varphi(r)}{\psi(T(r,f))}\right)\leq C_0 T(r,f),\eea
for all $r\not\in E_4$ with $\text{log mes}\;E_4<+\infty$. Then from (\ref{ey6}) and (\ref{ey6a}), we get
\bea\label{ey6b} \left|\frac{\partial^{I_n}(f(z))}{\partial^{I}(f(z))}\right|\leq D\left(\frac{T(r,f)}{r}\right)^{|I_n|-|I|},\eea
for $z\in \mathbb{C}^m$ such that $||z||\not\in [0,1]\cup E\cup E_4$ and $D$ is a constant.
Since $\rho=\rho(f)<+\infty$, for every $\varepsilon>0$, there exists $R>0$ such that $T(r,f)<r^{\rho+\varepsilon/2}$
for all $r\geq R$. Let us choose $\varepsilon>0$ such that $D<r^{(|I_n|-|I|)\varepsilon/2}$ for large values of $r$. Finally from (\ref{ey6b}), we get the desired result.
\end{proof}

\section{{\bf Maximum term and central index of holomorphic function in several complex variables}}
Given a point $c =(c _1,\ldots,c_m)\in \mathbb{C}^m$ and a positive real number $r_1,\ldots,r_m$ we put
\[U_{(r_1,\ldots,r_m)}(c)=\{z=(z_1,\ldots,z_m)\in\mathbb{C}^m: |z_k-c_k|<r_k, k=1,2,\ldots,m\}.\]

If $U_{r_k}(c_k)$ is the disk with centre $c_k$ and radius $r_k$ on the $z_k$-plane, then 
\[U_{(r_1,\ldots,r_m)}(c)=U_{r_1}(c_1)\times\ldots \times U_{r_m}(c_m).\]
 
We call $U_{(r_1,\ldots,r_m)}(c)$ the polydisk with centre $c$. Clearly 
\[\ol{U}_{(r_1,\ldots,r_m)}(c)=\{z=(z_1,\ldots,z_m)\in\mathbb{C}^m: |z_k-c_k|\leq r_k, k=1,2,\ldots,m\}\]
and $\ol{U}_r(c)=\ol{U}_{(r,\ldots,r)}(c)$. We denote by $C_k(c_k,r_k)$ the boundary of $U_{r_k}(c_k)$.
The product 
\[C_{(c)}^m(r_1,\ldots,r_m)=C_1(c_1,r_1)\times\ldots\times C_n(c_m,r_m)\]
is called the determining set of the polydisk $U_{(r_1,\ldots,r_m)}(c)$. $C_{(c)}^m(r_1,\ldots,r_m)$ is an $m$-dimensional torus. Clearly 
\[C_{(c)}^m(r)=C_1(c_1,r)\times\ldots\times C_m(c_m,r).\]

\smallskip
Let $f:\mathbb{C}^m\to \mathbb{C}$ be a holomorphic function. We define
\[M(r,f)=\max\limits_{z\in \mathbb{C}^m\langle r\rangle}|f(z)|,\]
where $\mathbb{C}^m\langle r\rangle=\{z\in\mathbb{C}^m:||z||=r\}$. If $f$ is non-constant, then $M(r,f)\to \infty$ as $r\to \infty$.
We can expand $f(z)$ to a Taylor series with multi-indices $\alpha$
\bea\label{eww} f(z)=\sum\limits_{|\alpha|=0}^{\infty}a_{\alpha}z^{\alpha}=\sum\limits_{\alpha_1,\ldots,\alpha_m=0}^{\infty} a_{\alpha_1\ldots\alpha_m}z_1^{\alpha_1}\ldots z_m^{\alpha_m}.\eea

By Cauchy's inequalities, we have
\bea\label{ew1}|a_{\alpha_1\ldots\alpha_m}|\leq \frac{\max\limits_{z\in C_{(0)}^m(r)} |f(z)|}{r^{|\alpha|}}.\eea

Obviously the power series defined by (\ref{eww}) is absolutely convergent in $C_{(0)}^m(r,\ldots,r)$, where $r$ is any positive real number. Consequently the power series
\bea\label{ew2}\sum\limits_{\alpha_1,\ldots,\alpha_m=0}^{\infty}|a_{\alpha_1\ldots\alpha_m}|r^{\alpha_1+\ldots+\alpha_m}\eea
converges for every $r>0$. Define
\bea\label{ew3}||a_{|\alpha|}||_{1}=\sum\limits_{\alpha_1+\ldots+\alpha_m=|\alpha|} |a_{\alpha_1\ldots\alpha_m}|.\eea

Therefore from (\ref{ew2}) and (\ref{ew3}), we conclude that the power series $\sum_{|\alpha|=0}^{\infty}||a_{|\alpha|}||_1\;r^{|\alpha|}$ converges for every $r>0$. Consequently for a given $r>0$, we have
\bea\label{ew4} \lim\limits_{|\alpha|\to\infty}||a_{|\alpha|}||_1 r^{|\alpha|}=0\eea
and so the maximum term
\[\tilde{\mu}(r)=\tilde{\mu}(r,f)=\max_{|\alpha|\geq 0} ||a_{|\alpha|}||_1\;r^{|\alpha|}\]
is well defined. Next we define
\[\tilde{\nu}(r)=\tilde{\nu}(r,f)=\max\{|\alpha|:||a_{|\alpha|}||_1\;r^{|\alpha|}=\tilde{\mu}(r,f)\}.\]

It is easy to verify that
\[||a_{|\alpha|}||_1\;r^{|\alpha|}\leq \tilde{\mu}(r,f)\;\;\text{for all}\;\;|\alpha|\geq 0\;\;\text{and}\;\; ||a_{|\alpha|}||_1\;r^{|\alpha|}<\tilde{\mu}(r,f)\;\;\text{for all}\;\;|\alpha|>\tilde{\nu}(r,f).\]

For $0<r<R$, we also have
\[ ||a_{\tilde{\nu}(R)}||_1\;R^{\tilde{\nu}(R)}\geq ||a_{\tilde{\nu}(r)}||_1\;R^{\tilde{\nu}(r)}
\;\;\text{and}\;\;||a_{\tilde{\nu}(r)}||_1\;r^{\tilde{\nu}(r)}\geq ||a_{\tilde{\nu}(R)}||_1\;r^{\tilde{\nu}(R)}.\]

Clearly, for a polynomial
\[P(z)=\sum\limits_{\alpha_1,\ldots,\alpha_m=0}^{n}a_{\alpha_1\ldots\alpha_m}z_1^{\alpha_1}\ldots z_m^{\alpha_m},\]
we have $\tilde{\mu}(r,P)=||a_{n}||_1\;r^{n}$ and $\tilde{\nu}(r,P)=n$. Also for $z\in \mathbb{C}^m_{(0)}(r)$, we can prove that
\bea\label{mm} (1-o(1))||a_n||_1r^n\leq |P(z)|\leq (1+o(1))||a_n||_1r^n.\eea

Henceforth we may assume that $f$ is a transcendental entire function while considering basic properties of the maximum term and the central index.
The following properties hold (see Jank and Volkmann \cite{JV1}, pp. 33-35):
\begin{enumerate}\item[(1)]$\tilde{\mu}(r,f)$ is strictly increasing for all $r$ sufficiently large, is continuous and tends to $+\infty$ as $r\rightarrow \infty$;
\item[(2)] $\tilde{\nu}(r,f)$ is increasing, piecewise constant, right-continuous and also tends to $+\infty$ as $r\rightarrow \infty$.
\end{enumerate}

\subsection{{\bf Relations between $\tilde{\mu}(r)$, $\tilde{\nu}(r)$ and $M(r)$}}

\begin{theo}\label{t3.1} Let $f:\mathbb{C}^m\to \mathbb{C}$ be a non-constant entire function. Then there exists a positive real number $R$ sufficiently large such that
\bea\label{ew5a}\tilde{\mu}(r,f)\leq M^m(\sqrt{m}r,f),\;\;\text{for all}\;\;r\geq R.\eea
\end{theo}

\begin{proof} By the definitions,
we have $\tilde{\mu}(r)=||a_{\tilde{\nu}(r)}||_1\;r^{\tilde{\nu}(r)}$. Now using (\ref{ew3}), we have
\bea\label{ew5}\tilde{\mu}(r)=||a_{\tilde{\nu}(r)}||_1\;r^{\tilde{\nu}(r)}=\sum\limits_{\alpha_1+\ldots+\alpha_m=\tilde{\nu}(r)} |a_{\alpha_1\ldots\alpha_m}|\;r^{\tilde{\nu}(r)}.\eea

Then from (\ref{ew1}), we get
\bea\label{ew6}|a_{\alpha_1\ldots\alpha_m}|\;r^{\tilde{\nu}(r)}\leq \max\limits_{z\in C_{(0)}^m(r,\ldots,r)} |f(z)|\leq \max\limits_{z\in\mathbb{C}^m\langle \sqrt{m}r\rangle}|f(z)|=M(\sqrt{m}r,f).\eea

But we know that the number of solutions of $\alpha_1+\ldots+\alpha_m=\tilde{\nu}(r)$ is $\binom {\tilde{\nu}(r)+m-1}{m-1}$.
Therefore from (\ref{ew5}) and (\ref{ew6}), we get
\bea\label{ew7}\tilde{\mu}(r)\leq \binom {\tilde{\nu}(r)+m-1}{m-1}M(\sqrt{m}r,f).\eea

Note that
\[\binom {\tilde{\nu}(r)+m-1}{m-1}= \frac{(\tilde{\nu}(r)+1)(\tilde{\nu}(r)+2) \cdots (\tilde{\nu}(r)+m-1)}{(m-1)!}.\]

Let us choose $(\alpha_1,\ldots,\alpha_m)\in\mathbb{Z}^m_+$ such that $|a_{\alpha_1\ldots\alpha_m}|\not=0$, where $\alpha_1+\ldots+\alpha_m=\tilde{\nu}(r)$. Now there exists a positive real number $R$ sufficiently large such that
\[\frac{\tilde{\nu}(r)+k}{\sqrt[m-1]{(m-1)!}}<|a_{\alpha_1\ldots\alpha_m}|\;r^{\tilde{\nu}(r)},\]
for all $r\geq R$, where $k = 1, 2, \dots, m-1$. Consequently, we have
\beas \binom {\tilde{\nu}(r)+m-1}{m-1}<|a_{\alpha_1\ldots\alpha_m}|^{m-1}\;r^{(m-1)\tilde{\nu}(r)},\eeas
for all $r\geq R$ and so from (\ref{ew6}), we get
\bea\label{ew8} \binom {\tilde{\nu}(r)+m-1}{m-1}<M^{m-1}(\sqrt{m}r,f),\eea
for all $r\geq R$. Finally from (\ref{ew7}) and (\ref{ew8}), we have $\tilde{\mu}(r)\leq M^m(\sqrt{m}r,f)$,
for all $r\geq R$.
\end{proof}

\begin{theo}\label{t3.2} Let $f:\mathbb{C}^m\to \mathbb{C}$ be a non-constant entire function. Then for $0<r<R$
\bea\label{ew9a} M(r,f)\leq \tilde{\mu}(r,f)\left(\tilde{\nu}(R,f)+\frac{R}{R-r}\right).\eea
\end{theo}

\begin{proof} Suppose $0<r<R$. By the definitions of $\tilde{\mu}(r)$ and $\tilde{\nu}(r)$, we have
\bea\label{ew9} \tilde{\mu}(r)=||a_{\tilde{\nu}(r)}||_1\;r^{\tilde{\nu}(r)}\geq ||a_{\tilde{\nu}(R)}||_1\;r^{\tilde{\nu}(R)}.\eea

Now for $||z||=r$, we get from (\ref{eww}) that
\beas |f(z)|\leq \sum\limits_{\alpha_1,\ldots,\alpha_m=0}^{\infty}|a_{\alpha_1\ldots\alpha_m}| r^{|\alpha|}=\sum\limits_{|\alpha|=0}^{\infty}||a_{|\alpha|}||_1\;r^{|\alpha|}\eeas
and so from (\ref{ew9}), we have
\beas M(r,f)\leq \sum_{|\alpha|=0}^{\infty}||a_{|\alpha|}||_1\;r^{|\alpha|}
&=& \sum_{|\alpha|=0}^{\tilde{\nu}(R)-1}||a_{|\alpha|}||_1\;r^{|\alpha|}+\sum_{|\alpha|=\tilde{\nu}(R)}^{\infty}||a_{|\alpha|}||_1\;r^{|\alpha|}\\
 &\leq &
\tilde{\mu}(r)\tilde{\nu}(R)+\sum_{|\alpha|=\tilde{\nu}(R)}^{\infty}||a_{|\alpha|}||_1\;r^{|\alpha|}\frac{||a_{\tilde{\nu}(r)}||_1\;r^{\tilde{\nu}(r)}}{||a_{\tilde{\nu}(R)}||_1\;r^{\tilde{\nu}(R)}}\\
&\leq &
\tilde{\mu}(r)\tilde{\nu}(R)+\tilde{\mu}(r)\sum_{|\alpha|=\tilde{\nu}(R)}^{\infty}\frac{||a_{|\alpha|}||_1\;R^{|\alpha|}}{||a_{\tilde{\nu}(R)}||_1\;R^{\tilde{\nu}(R)}}\left(\frac{r}{R}\right)^{|\alpha|-\tilde{\nu}(R)}\\
&\leq &
\tilde{\mu}(r)\tilde{\nu}(R)+\tilde{\mu}(r)\sum_{|\alpha|=\tilde{\nu}(R)}^{\infty}\left(\frac{r}{R}\right)^{|\alpha|-\tilde{\nu}(R)}
\\&=& \tilde{\mu}(r)\left(\tilde{\nu}(R)+\frac{R}{R-r}\right).\eeas
\end{proof}

Now we recall the following lemma.
\begin{lem}\cite[Lemma 2.5.24]{NW}\label{l3.2}
Let $f:\mathbb{C}^m\to \mathbb{C}$ be an entire function. Then for $0<r<R$,
\[T(r,f)\leq \log^+\max_{||z||=r}|f(z)|\leq \frac{1-\left(\frac{r}{R}\right)^2}{\left(1-\frac{r}{R}\right)^{2m}}T(R,f).\]
\end{lem}

From Lemma \ref{l3.2}, we can prove that
\[\rho(f):=\limsup _{r \rightarrow \infty} \frac{\log^+ T(r, f)}{\log r}=\limsup _{r \rightarrow \infty} \frac{\log^+ \log^+ M(r, f)}{\log r}.\]

\begin{theo}\label{t3.3} Let $f:\mathbb{C}^m\to \mathbb{C}$ be a non-constant entire function. Then
\beas \rho=\rho(f)=\varlimsup\limits_{r\to\infty} \frac{\log^+ \log^+ \tilde{\mu}(r)}{\log r}=\varlimsup\limits_{r \to \infty} \frac{\log^+ \tilde{\nu}(r)}{\log r}.\eeas
\end{theo}

\begin{proof} Let
\[\hat \rho_1=\varlimsup\limits_{r\to\infty} \frac{\log^+ \log^+ \tilde{\mu}(r)}{\log r }\;\;\text{and}\;\;\hat\rho_2=\varlimsup\limits_{r \to \infty} \frac{\log^+ \tilde{\nu}(r)}{\log r}.\]

\smallskip
First we suppose $f$ is a non-constant polynomial. Let
\[f(z)=\sum\limits_{\alpha_1,\ldots,\alpha_m=0}^{n}a_{\alpha_1\ldots\alpha_m}z_1^{\alpha_1}\ldots z_m^{\alpha_m}\]
such that $||a_n||_1\neq 0$. Clearly $\tilde{\mu}(r,f)=||a_{n}||_1\;r^{n}$ and $\tilde{\nu}(r,f)=n$.
It is easy to prove that $\hat\rho_1=0=\hat\rho_2$. On the other hand we have $\rho(f)=0$. Therefore $\rho=\hat\rho_1=\hat\rho_2$.

\smallskip
Next we suppose $f(z)$ is a transcendental entire function. By the definitions of $\tilde{\mu}(r)$ and $\tilde{\nu}(r)$, we have
$\tilde{\mu}(r) =||a_{\tilde{\nu}(r)}||_1\;r^{\tilde{\nu}(r)}$. On the other hand, from (\ref{ew4}), we get 
\[\lim\limits_{|\alpha|\to\infty}||a_{|\alpha|}||_1=0.\]

Consequently for sufficiently large $r$, we have
$\tilde{\mu}(r)=||a_{\tilde{\nu}(r)}||_1\;r^{\tilde{\nu}(r)}\leq r^{\tilde{\nu}(r)}$, from which, we immediately get $\hat\rho_1\leq \hat\rho_2$.
For $0<r<R$, we have
\[\left(\frac{R}{r}\right)^{\tilde{\nu}(r)}=\frac{||a_{\tilde{\nu}(r)}||_1\;R^{\tilde{\nu}(r)}}{||a_{\tilde{\nu}(r)}||_1\;r^{\tilde{\nu}(r)}}\leq \frac{\tilde{\mu}(R)}{\tilde{\mu}(r)}\]
and so
\bea\label{ew10} \tilde{\nu}(r) \log \left(\frac{R}{r}\right)\leq \log \tilde{\mu}(R)-\log \tilde{\mu}(r) \leq \log \tilde{\mu}(R).\eea

Putting $R=3r$ in (\ref{ew10}), we get 
\[\tilde{\nu}(r)\leq \frac{1}{\log 3}\log \tilde{\mu}(3r),\]
from which we immediately get $\hat\rho_2\leq \hat\rho_1$. Therefore $\hat\rho_1=\hat\rho_2$.

On the other hand, from (\ref{ew5a}), we see that 
\[\tilde{\mu}(r)\leq M^m(\sqrt{m}r,f)\]
holds for sufficiently large values of $r$.
Therefore one can easily prove that $\hat\rho_1 \leq \rho$. Now we want to show $\rho \leq \hat\rho_1$. If $\hat\rho_1=\infty$, there is nothing to prove. So we assume $\hat\rho_1<+\infty$. Clearly $\hat\rho_2<+\infty$ and so
\bea\label{ew11}\tilde{\nu}(r)<r^{\hat\rho_2+\varepsilon}\eea
holds for sufficiently large values of $r$, where $\varepsilon>0$.

Putting $R=2r$ in (\ref{ew9a}) and using (\ref{ew11}), we get
\[M(r) \leq \tilde{\mu}(r)\left((2r)^{\rho_2+\varepsilon} + 2\right) \leq \tilde{\mu}(r) r^{\hat\rho_2+2\varepsilon},\]
from which we immediately get $\rho \leq \hat\rho_1$. Consequently $\rho=\hat\rho_1$. Finally $\rho=\hat\rho_1=\hat\rho_2$.
\end{proof}

From the proof of Theorem \ref{t3.3}, we can prove that
\begin{theo}\label{t3.3a} Let $f:\mathbb{C}^m\to \mathbb{C}$ be a non-constant holomorphic function. Then
\beas \rho_1=\rho_1(f)=\varlimsup\limits_{r\to\infty} \frac{\log^+\log^+ \log^+ M(r)}{\log r }=\varlimsup\limits_{r \to \infty} \frac{\log^+\log^+ \tilde{\nu}(r)}{\log r},\eeas
where $\rho_1(f)$ is called the hyper-order of $f$.
\end{theo}

In applications to complex differential equations, the following result is of major importance. Actually following theorem forms a powerful tool for order considerations of entire solutions of linear (and algebraic) differential equations.

\begin{theo3A} \cite[Satz 21.3]{JV1} Let $f:\mathbb{C}\to \mathbb{C}^n$ be vector-valued function, whose components $f_1,\ldots,f_n$ are  transcendental entire functions. Let $0<\delta<\frac{1}{4}$ and $z\in\mathbb{C}$ be such that $|z|=r$ and
\beas ||f(z)||>M(r)[v(r)]^{-\frac{1}{4}+\delta}\eeas
holds. Then there exists a set $F\subset \mathbb{R}_+$ of finite logarithmic measure such that
\beas f^{(m)}(z)=\left(\frac{v(r)}{z}\right)^m (E+o(1))f(z), \;\; r \notin F \eeas
holds for all $m \in \mathbb{N}$, where $E$ is the $n$-th identity matrix, $o(1)$ is a matrix that tends to $0$ as $|z| \rightarrow \infty$, $|z| \notin F$.
\end{theo3A}

The main purpose of this section is to study the validation of Theorem 3.A for the case when $f:\mathbb{C}^m\to \mathbb{C}$ is a transcendental entire function. As a outcome of our study, we get the following result.

\begin{theo}\label{t3.4} Let $f:\mathbb{C}^m\to \mathbb{C}$ be a transcendental entire function, $I=(i_1,\ldots,i_m)\in \mathbb{Z}^m_+$ be a multi-index and let $L(z)=a_1z_1+\cdots+a_mz_m$, where $a_i\in\mathbb{C}\backslash \{0\}$. Let $z\in  C^m_{(0)}(r)$ such that
\[|f(z)|>M(\sqrt{m}r) \tilde{\nu}(r)^{-\frac{1}{4}+\delta}\]
holds for $0<\delta<\frac{1}{4}$. Then there exists a set $E\subset (1,\infty)$ of finite logarithmic measure such that
\beas \frac{\partial^{I}(f(z))}{f(z)}=\left(\frac{\tilde{\nu}(r)}{L(z)}\right)^{|I|} (1+o(1)),\eeas
holds for all large values of $r=||z||\not\in [0,1]\cup E$, where $z\in C^m_{(0)}(r)$.
\end{theo}

\begin{proof}[{\bf Proof of Theorem \ref{t3.4}}]
Let $f:\mathbb{C}^m\to \mathbb{C}$ be a transcendental entire function. Then $f(z)$ can be expanded to a power series in $\ol{U}_{r}(0)$ as in (\ref{eww}). We set $z_j=re^{i\theta_j}\;(0\leq\theta_j\leq 2\pi)$, where $j=1,2,\ldots,m$ and consider the integration of $\left|f\left(re^{i\theta_1},\ldots,re^{i\theta_m}\right)\right|^2$. Noting that $\int_0^{2\pi} e^{i(\beta-\gamma)\theta}d\theta=2\pi \delta_{\beta,\gamma}$ (Kronecker's symbol), we have
\beas \sum_{|\alpha|=0}^{\infty} ||a_{|\alpha|}||_1^2\; r^{2|\alpha|}&=&\frac{1}{(2\pi)^m} \int_{0}^{2\pi} \cdots \int_{0}^{2\pi} \left|f\left(re^{i\theta_1},\ldots,re^{i\theta_m}\right)\right|^2 \, d\theta_1 \cdots d\theta_m\\&\leq&
\max_{||z||=\sqrt{m}r}|f(z)|^2=M^2(\sqrt{m}r,f),\eeas
from which we deduce that
\bea\label{ew14} M(\sqrt{m}r,f)>r^{\beta},\eea
for all large values of $r$ such that $z\in C^m_{(0)}(r)$, where $\beta$ is any real number.

Now from (\ref{ew10}), we have 
\[\tilde{\nu}(r) \log \left(\frac{R}{r}\right)+\log \tilde{\mu}(r)\leq \log \tilde{\mu}(R).\]

Since
$\tilde{\mu}(r)=||a_{\tilde{\nu}(r)}||_1\;r^{\tilde{\nu}(r)}$, we can deduce that
\bea\label{eeww1} \tilde{\nu}(r)\left(\log \left(\frac{R}{r}\right)+1\right)\leq \log \tilde{\mu}(R),\eea
$0<r<R$. Let us choose
\bea\label{eeww2a}\psi(r)=(\log r)^{1+\varepsilon}\;(\varepsilon >0), \;\varphi(r)=r\eea
and
\bea\label{eeww2} R=r+\frac{\varphi(r)}{\psi(\tilde{\mu}(r))}.\eea

By applying Lemma \ref{l3.1} to the increasing function $\tilde{\mu}(r)$, we have
\bea\label{eeww3}\tilde{\mu}(R)=\tilde{\mu}\left(r+\frac{\varphi(r)}{\psi(\tilde{\mu}(r))}\right)\leq C_1 \tilde{\mu}(r),\eea
for all large $r\not\in E_{1}$ with $\text{log mes}\;E_{1}<+\infty$. Now using (\ref{eeww2}) and (\ref{eeww3}) to (\ref{eeww1}), we get
\bea\label{eew3} \tilde{\nu}(r)\leq \log C_1\tilde{\mu}(r),\eea
for all large $r$ outside the set $E_{1}$. Now using (\ref{ew5a}) to (\ref{eew3}), we get
\bea\label{eew4} \tilde{\nu}(r)\leq \log C_1\tilde{\mu}(r)\leq m\log C_1 M(\sqrt{m}r)\leq m \sqrt[4]{M(\sqrt{m}r)},\eea
for all large $r$ outside the set $E_{1}$ with $\text{log mes}\;E_{1}<+\infty$.

Suppose $r>1$. Now for $z\in C^m_{(0)}(r)$, we get from (\ref{eww}) that
\bea\label{eew5} |f(z)|\leq \sum\limits_{|\alpha|=0}^{\infty}||a_{|\alpha|}||_1\;r^{|\alpha|}.\eea

Clearly 
\[\lim\limits_{|\alpha|\to \infty}||a_{|\alpha|}||_1\;r^{|\alpha|}=0\]
and so 
\[\lim\limits_{|\alpha|\to \infty}\sqrt[|\alpha|]{||a_{|\alpha|}||_1\;r^{|\alpha|}}=0.\]

Let
\bea\label{eeww5a} R=r+\frac{r}{(\log M(r,f))^{1+\varepsilon_1}},\eea
where $0<\varepsilon_1<\frac{1}{4}$. Clearly $R>r$. Now there exists a positive integer $n_0$ such that
\bea\label{eeww5b} ||a_{|\alpha|}||_1\;r^{|\alpha|}<\left(\frac{r}{R}\right)^{|\alpha|},\;\;\text{for all}\;\;|\alpha|\geq n_0.\eea

Note that
\[\sum\limits_{|\alpha|=0}^{n_0-1}||a_{|\alpha|}||_1\;r^{|\alpha|}\leq  \left(\sum\limits_{|\alpha|=0}^{n_0-1}||a_{|\alpha|}||_1\right)r^{n_0-1}\]
and so from (\ref{ew14}), we get
\bea\label{eew6}\sum\limits_{|\alpha|=0}^{n_0-1}||a_{|\alpha|}||_1\;r^{|\alpha|}\leq  \left(\sum\limits_{|\alpha|=0}^{n_0-1}||a_{|\alpha|}||_1\right)r^{n_0}<\sqrt[4]{M(\sqrt{m}r)}.\eea

Now using (\ref{eew4}) to (\ref{eew6}), we have
\bea\label{eew7}\sum\limits_{|\alpha|=0}^{n_0-1}||a_{|\alpha|}||_1\;r^{|\alpha|}<m^2 \sqrt{M(\sqrt{m}r)}\;\tilde{\nu}(r)^{-1},\eea
for all large $r$ outside the set $E_{1}$ with $\text{log mes}\;E_{1}<+\infty$.
Also from (\ref{eeww5b}), we have
\bea\label{eeww8} \sum_{|\alpha|=n_0}^{\infty}||a_{|\alpha|}||_1\;r^{|\alpha|}
\leq \sum_{|\alpha|=n_0}^{\infty}\left(\frac{r}{R}\right)^{|\alpha|}\leq \sum_{|\alpha|=0}^{\infty}\left(\frac{r}{R}\right)^{|\alpha|}
=\frac{R}{R-r}.\eea

Now using (\ref{eeww5a}) to (\ref{eeww8}), we get
\bea\label{eeww9} \sum_{|\alpha|=n_0}^{\infty}||a_{|\alpha|}||_1\;r^{|\alpha|}
\leq \left(\log M(r)\right)^{1+\varepsilon_1}\leq \left(\sqrt[4]{M(\sqrt{m}r)}\right)^{1+\varepsilon_1}.\eea

Therefore using (\ref{eew4}) to (\ref{eeww9}), we obtain
\bea\label{eeww10}\sum\limits_{|\alpha|=n_0}^{\infty}||a_{|\alpha|}||_1\;r^{|\alpha|}<m^2 \left(\sqrt{M(\sqrt{m}r)}\right)^{1+\varepsilon_1}\;\tilde{\nu}(r)^{-(1+\varepsilon_1)},\eea
for all large $r$ outside the set $E_{1}$.
Consequently using (\ref{eew7}) and (\ref{eeww10}) to (\ref{eew5}), we get
\bea\label{eeww11} |f(z)|\leq \sum\limits_{|\alpha|=0}^{\infty}||a_{|\alpha|}||_1\;r^{|\alpha|}<m^2 \left(\sqrt{M(\sqrt{m}r)}\right)^{1+\varepsilon_1}\;\tilde{\nu}(r)^{-\frac{1}{4}},\eea
for all large $r\not\in E_{1}$ with $\text{log mes}\;E_{1}<+\infty$, where $z\in C^m_{(0)}(r)$.

On the other hand, from (\ref{eyy21}), we get
\bea\label{ew15} \left|\partial^{J}\left(\frac{\partial_{z_i}(f(z))}{f(z)}\right)\right|\leq
\frac{(m+n-1)!2^{m+n+1}}{(m-1)!}
\left(\frac{T(\alpha^2 r,f)}{r^n}+\frac{n_{f}(\alpha^2 r)}{r^n}\right)
,\eea
for $z\in \mathbb{C}^m\langle r\rangle\backslash A$, where $n_{f}(r)=n_{\mu_{f}^0}(r)$, $\alpha>1$ is a real constant and $|J|+1=|I|=n$.

Let $E_2=\tau(A)\backslash  [0,1]$. From (\ref{q1}), we have $\text{log mes}\;E_2<+\infty$. Then from (\ref{ew15}), we get
\bea\label{ew16} \left|\partial^{J}\left(\frac{\partial_{z_i}(f(z))}{f(z)}\right)\right|\leq
\frac{(m+n-1)!2^{m+n+1}}{(m-1)!}\left(\frac{T(\alpha^2 r,f)}{r^n}+\frac{n_{f}(\alpha^2 r)}{r^n}\right)
,\eea
for $z\in \mathbb{C}^m$ such that $||z||\not\in [0,1]\cup E_2$. By Lemma \ref{l2.4}, we have
\bea\label{ew17} n_{\mu^0_{f}}(r)\leq (1+o(1))\frac{T(\alpha^2 r,f)}{\log \alpha}\leq \frac{2}{\log \alpha}T(\alpha^2 r,f),\eea
as $r\rightarrow \infty$.
Consequently from (\ref{ew16}) and (\ref{ew17}), we get
\bea\label{ew18} \left|\partial^{J}\left(\frac{\partial_{z_i}(f(z))}{f(z)}\right)\right|\leq
\frac{(m+n-1)!2^{m+n+1}(2+\log \alpha)}{(m-1)!\log \alpha}\frac{T(\alpha^2 r,f)}{r^n}
,\eea
for $z\in \mathbb{C}^m$ such that $||z||\not\in [0,1]\cup E_2$.
Let us choose
\[\alpha^2=1+\frac{\varphi(r)}{r\psi(M(r))},\]
where $\psi(r)$ and $\varphi(r)$ are given as in (\ref{eeww2a}). Now applying Lemma \ref{l3.1} to the increasing function $M(r)$, we have
\bea\label{ew19}M(\alpha^2 r)=M\left(r+\frac{\varphi(r)}{\psi(M(r))}\right)\leq C_1 M(r),\eea
for all $r\not\in E_3$ with $\text{log mes}\;E_3<+\infty$. Using Lemma \ref{l3.2}, we get from (\ref{ew18}) and (\ref{ew19}) that
\bea\label{ew20} \left|\partial^{J}\left(\frac{\partial_{z_i}(f(z))}{f(z)}\right)\right|\leq
\frac{(m+n-1)!2^{m+n+1}(2+\log \alpha)C_1}{(m-1)!\log \alpha}\frac{\log M(\sqrt{m}r)}{r^n}\leq C_2\frac{\sqrt[4]{M(\sqrt{m}r)}}{r^n}
,\eea
for all $z\in \mathbb{C}^m$ such that $||z||\not\in [0,1]\cup E_2\cup E_3$ and $C_2=\frac{(m+n-1)!2^{m+n+1}(2+\log \alpha)C_1}{(m-1)!\log \alpha}<+\infty$.

If $z\in C^m_{(0)}(r)$, then $|L(z)|\leq r ||a||$.
Let $E_4=\tau(\text{supp}\;\mu^0_L)\backslash  [0,1]$. Clearly (\ref{q1}) yields $\text{log mes}\;E_4<+\infty$. Let $E=E_1\cup E_2\cup E_3\cup E_4$.
Now using (\ref{eeww11}) and (\ref{ew20}), we get
\bea\label{ew22} &&\left|f(z)\left(\frac{L(z)}{\tilde{\nu}(r)}\right)^n\partial^{J}\left(\frac{\partial_{z_i}(f(z))}{f(z)}\right)-f(z)\right|\\&\leq&
|f(z)| \left(\frac{|L(z)|}{\tilde{\nu}(r)}\right)^n\left|\partial^{J}\left(\frac{\partial_{z_i}(f(z))}{f(z)}\right)\right|+|f(z)|\nonumber\\&\leq&
C_2||a||^n\frac{\sqrt[4]{M(\sqrt{m}r)}}{\tilde{\nu}^n(r)}m^2 \left(\sqrt{M(\sqrt{m}r)}\right)^{1+\varepsilon_1}\;\tilde{\nu}(r)^{-\frac{1}{4}}+m^2 \left(\sqrt{M(\sqrt{m}r)}\right)^{1+\varepsilon_1}\;\tilde{\nu}(r)^{-\frac{1}{4}}\nonumber\\&\leq&
\left(m^2C_2||a||^n \frac{1}{M(\sqrt{m}r)^{\frac{1-2\varepsilon_1}{4}}\tilde{\nu}^n(r)}+\frac{1}{M(\sqrt{m}r)^{\frac{1-\varepsilon_1}{2}}}\right) M(\sqrt{m}r)\;\tilde{\nu}(r)^{-\frac{1}{4}}\nonumber
,\eea
for all values of $r=||z||\not\in [0,1]\cup E$.
Since $M(r)\to +\infty$ as $r\to +\infty$, from (\ref{ew22}), we get
\bea\label{ew23}\left|f(z)\left(\frac{L(z)}{\tilde{\nu}(r)}\right)^n\partial^{J}\left(\frac{\partial_{z_i}(f(z))}{f(z)}\right)-f(z)\right|\leq M(\sqrt{m}r)\;\tilde{\nu}(r)^{-\frac{1}{4}},\eea
for all large values of $r$ such that $r=||z||\not\in [0,1]\cup E$, where $z\in C^m_{(0)}(r)$.
Also we have
\bea\label{ew24}|f(z)|>M(\sqrt{m}r) \tilde{\nu}(r)^{-\frac{1}{4}+\delta},\eea
where $z\in C^m_{(0)}(r)$. Now from (\ref{ew23}) and (\ref{ew24}), we see that
\bea\label{ew25}\left|\frac{1}{|f(z)|}\left(f(z)\left(\frac{L(z)}{\tilde{\nu}(r)}\right)^n\partial^{J}\left(\frac{\partial_{z_i}(f(z))}{f(z)}\right)-f(z)\right)\right|\leq \tilde{\nu}(r)^{-\delta},\eea
for all large values of $r=||z||\not\in [0,1]\cup E$, where $z\in C^m_{(0)}(r)$.
From (\ref{ew25}), we can write
\bea\label{ew26} f(z)\left(\frac{L(z)}{\tilde{\nu}(r)}\right)^n\partial^{J}\left(\frac{\partial_{z_i}(f(z))}{f(z)}\right)-f(z)=|f(z)|\eta_1(z),\eea
where $\eta_1(z)\to 0$ as $||z||\to \infty$, where $r=||z||\not\in [0,1]\cup E$.
Define
\beas
\eta(z):=\left\{
\begin{array}{ll}
\eta_1(z) \frac{\left|f(z)\right|}{f(z)}, &\text{if}\; f(z) \neq 0 \\
\eta_1(z), & \text{if}\; f(z)=0.
\end{array}
\right.
\eeas

Then $\eta(z) \rightarrow 0$ as $||z|| \rightarrow \infty$ and so from (\ref{ew26}), we obtain
\bea\label{ew27} \partial^{J}\left(\frac{\partial_{z_i}(f(z))}{f(z)}\right)=\left(\frac{\tilde{\nu}(r)}{L(z)}\right)^n (1+o(1)),\eea
for all large values of $r$ such that $r=||z||\not\in [0,1]\cup E$, where $z\in C^m_{(0)}(r)$.

If $|J|=0$, then from (\ref{ew27}), we get
\bea\label{ew28} \frac{\partial_{z_i}(f(z))}{f(z)}=\frac{\tilde{\nu}(r)}{L(z)} (1+o(1)),\eea
for all values of $r$ such that $r=||z||\not\in [0,1]\cup E$, where $z\in C^m_{(0)}(r)$.
Let $I=(i_{11},\ldots,i_{1i},\ldots,i_{1m})$, where $i_{1i}=1$ and $i_{1j}=0$ for $j\neq i$. Clearly $|I|=1$ and so from (\ref{ew28}), we have
\beas \frac{\partial^{I}(f(z))}{f(z)}=\frac{\tilde{\nu}(r)}{L(z)} (1+o(1)),\eeas
for all large values of $r$ such that $r=||z||\not\in [0,1]\cup E$, where $z\in C^m_{(0)}(r)$.

\medskip
Suppose $|J|=1$. Let $\partial^J(f)=\partial_{z_j}(f)$ for $j\in\mathbb{Z}[1,m]$ and $I_2=(i_{21},\ldots,i_{2i},\ldots, i_{2j},\ldots,i_{2m})$, where $i_{2i}=1$, $i_{2j}=1$ and $i_{2k}=0$ for $i,j\neq k$. Clearly $|I|=2$ and $\partial^{I}(f(z))=\partial_{z_j}(\partial_{z_i}(f(z)))$.
Now from (\ref{eyy23}) and (\ref{ew27})-(\ref{ew28}), we can easily conclude that
\bea\label{ew30} \frac{\partial^{I}(f(z))}{f(z)}=\frac{\partial_{z_j}(\partial_{z_i}(f(z)))}{f(z)}=\left(\frac{\tilde{\nu}(r)}{L(z)}\right)^2 (1+o(1)),\eea
for all large values of $r$ such that $r=||z||\not\in [0,1]\cup E$, where $z\in C^m_{(0)}(r)$.

\medskip
Suppose $|J|=2$. Let $\partial^J(f)=\partial_{z_k}(\partial_{z_l}(f))$ for any $k,l\in\mathbb{Z}[1,m]$. Assume that $i\leq k\leq l$. Let
$I_3=(i_{31},\ldots,i_{3i},\ldots,i_{3k},\ldots,i_{3l},\ldots,i_{3m})$, where $i_{3i}=1$, $i_{3k}=1$, $i_{3l}=1$ and $i_{3j}=0$ for $i,k,l\neq j$. Clearly $|I|=3$ and $\partial^{I}(f(z))=\partial_{z_l}(\partial_{z_k}(\partial_{z_i}(f(z))))$.
Using (\ref{eyy25}) and (\ref{ew27})-(\ref{ew30}), we get
\beas \frac{\partial^{I}(f(z))}{f(z)}=\frac{\partial_{z_l}(\partial_{z_k}(\partial_{z_i}(f(z))))}{f(z)}=\left(\frac{\tilde{\nu}(r)}{L(z)}\right)^3 (1+o(1)),\eeas
for all large values of $r$ such that $r=||z||\not\in [0,1]\cup E$, where $z\in C^m_{(0)}(r)$.
By repeating this process, it can be deduced from finite induction that
\beas \frac{\partial^{I}(f(z))}{f(z)}=\left(\frac{\tilde{\nu}(r)}{L(z)}\right)^{|I|} (1+o(1)),\eeas
for all large values of $r$ such that $r=||z||\not\in [0,1]\cup E$, where $z\in C^m_{(0)}(r)$.
\end{proof}

\section{{\bf Results on the partial differential equations $\partial^I(f(z))-e^{P(z)}f(z)=Q(z)$}}
By making use of Lemma \ref{l2.3}, it is easy to see that if $P(z)$ is non-constant, then every non-zero solution $f(z)$ of the equation 
\[\partial^I(f(z))-e^{P(z)}f(z)=0\]
is an entire function of order $\rho(f)=+\infty$. Let us consider non-homogeneous partial differential equation
\bea\label{e4.1} \partial^I(f(z))-e^{P(z)}f(z)=Q(z),\eea
where $P(z)$ and $Q(z)$ are entire functions in $\mathbb{C}^m$. If $m=1$ and $|I|=k$, then (\ref{e4.1}) gives
\bea\label{e4.2} f^{(k)}(z)+e^{P(z)}f(z)=Q(z),\;z\in\mathbb{C}.\eea

\smallskip
In 1998, Gundersen and Yang \cite{GY1} proved that if $P$ is a non-constant polynomial and if $Q\equiv 1$, then every solution $f$ of (\ref{e4.2}) is an entire function of infinite order. It was shown later in \cite{c1} that, when $Q$ is a non-zero polynomial and $P$ is a non-constant polynomial, every solution of (\ref{e4.2}) has infinite order and $\rho_1(f)$ is a positive integer such that $\rho_1(f)\leq \deg(P)$. In 2016, Cao \cite{c2} showed that if $P$ is a non-constant polynomial and if $f$ be a nonzero entire solution of (\ref{e4.2}), where $Q$ is an entire function that is ``small'' with respect to $f$, then $\rho_1(f)=\deg(P)$.

\smallskip
In this section, we continue to consider (\ref{e4.1}) with the case when $Q$ is an entire function that is ``small'' with respect to the solutions and $P$ is a non-constant polynomial.

\begin{theo}\label{t4.1} Let $P$ be a non-constant polynomial and let $f$ be a non-zero entire solution of (\ref{e4.1}), where $Q$ is an entire function that is ``small'' with respect to $f$. Then $\rho_1(f)=\deg(P)$.
\end{theo}

\begin{proof} By the given condition, we can say that any solution $f(\not \equiv 0)$ of (\ref{e4.1}) is transcendental. Now by Theorem \ref{t3.4}, there exists a subset $E\subset (1,+\infty)$ with finite logarithmic measure such that for $z_r\in \mathbb{C}^m_{(0)}(r)$ satisfying $||z_r||=r\not\in [0,1]\cup E$ and $M(\sqrt{m}r,f)=|f(z_r)|$, we have
\bea\label{m2} \frac{\partial^I f(z_r)}{f(z_r)}=\left(\frac{\tilde{\nu}(r)}{L(z_r)}\right)^{|I|}(1+o(1))\eea
as $r\to \infty$, where $L(z)=\sum_{j=1}^m a_jz_j$ with $a_j\in\mathbb{C}\backslash \{0\}$.
Since $Q$ is an entire function that is ``small'' with respect to $f$, $M(\sqrt{m}r,f)$ increases faster than $M(\sqrt{m}r,Q)$ and so
\bea\label{m3} \lim\limits_{r\rightarrow +\infty}\left|\frac{Q(z_r)}{f(z_r)}\right|\leq \lim\limits_{r\rightarrow +\infty}\frac{M(\sqrt{m}r,Q)}{M(\sqrt{m}r,f(z_r))}=0,\;\;\text{i.e.,}\;\;\frac{|Q(z)|}{|f(z)|}=o(1) \eea
for sufficiently large $||z_r||=r\not\in [0,1]\cup E$. Let $\deg(P)=n$.
Now from (\ref{e4.1}), we get
\bea\label{m5} \left|\frac{\partial^I(f(z))}{f(z)}\right|\leq \frac{|Q(z)|}{|f(z)|}+\left|e^{P(z)}\right|.\eea

Note that for $z_r\in \mathbb{C}^m_{(0)}(r)$, we have $|L(z_r)|\leq ||a|| r$, where $||a||=\sum_{j=1}^m|a_j|$. Now substituting (\ref{mm}) and (\ref{m2})-(\ref{m3}) into (\ref{m5}), we have
\beas |I| \log \tilde{\nu}(r) &\leq& \log \left( \frac{|Q(z_r)|}{|f(z_r)|}+\left|e^{P(z_r)}\right|\right)+|I|\log r +|I|\log ||a||+O(1)
\\&\leq& |P(z_r)|+|I|\log r+O(1)\leq ||a_n||_1r^n(1+ o(1))+|I|\log r+O(1)
\eeas
and so 
\[\log \log \tilde{\nu}(r)\leq n\log r+\log \log r + O(1)\]
for sufficiently large $r = ||z_r|| \not\in [0,1]\cup E$. Then by Theorem \ref{t3.3a}, we have 
\[\rho_1(f)\leq n=\deg(P).\]

\smallskip
Taking the principal branch of the logarithm, we deduce from (\ref{e4.1}) that
\bea\label{m7}P(z)=\log \left( \frac{\partial^I(f(z))}{f(z)} - \frac{Q(z)}{f(z)} \right)
=\log \left| \frac{\partial^I(f(z))}{f(z)}-\frac{Q(z)}{f(z)}\right|+ i\arg \left(\frac{\partial^I(f(z))}{f(z)}-\frac{Q(z)}{f(z)}\right).
\eea

\smallskip
Substituting (\ref{mm}) and (\ref{m2})-(\ref{m3}) into (\ref{m7}), we get
\beas ||a_n||_1 r^n(1 - o(1))\leq |P(z_r)|&\leq& \log \left|\frac{\partial^I(f(z_r))}{f(z_r)}\right|+\log \left(\left|\frac{Q(z_r)}{f(z_r)}\right|+e\right)+O(1)\\&\leq& |I| \log \tilde{\nu}(r)+O(1)
\eeas
and so 
\[n \log r \leq \log \log \tilde{\nu}(r)-\log \log r+O(1)\]
for sufficiently large $r=||z_r||\not\in [0,1]\cup E$. Then Theorem \ref{t3.3a} gives 
\[\deg(P) = n \leq \rho_1(f).\]

Hence $\rho_1(f)=\deg(P)$.
\end{proof}

\medskip
We know that if $\rho(f)<+\infty$, then $\rho_1(f)=0$. So from Theorem \ref{t4.1}, we get the following.
\begin{cor}\label{c4.1} Let $f:\mathbb{C}^m\to\mathbb{C}$ be a non-constant entire function of finite order. If $\partial^I(f)$ and $f$ share one finite value $a$ CM, then $\partial^I(f)-a=c(f-a)$ for some constant $c\in\mathbb{C}\backslash \{0\}$.
\end{cor}

Obviously Corollary \ref{c4.1} extends Theorem 1 \cite{GY1} to the case of higher dimensions.

\vspace{0.1in}
{\bf Compliance of Ethical Standards:}\par

{\bf Conflict of Interest.} The authors declare that there is no conflict of interest regarding the publication of this paper.\par

{\bf Data availability statement.} Data sharing not applicable to this article as no data sets were generated or analysed during the current study.

\end{document}